\documentclass[lettersize,journal]{IEEEtran}
\usepackage{amsmath,amsfonts, amssymb}
\usepackage{algorithmic}
\usepackage{algorithm}
\usepackage{array}
\usepackage[caption=false,font=normalsize,labelfont=sf,textfont=sf]{subfig}
\usepackage{textcomp}
\usepackage{stfloats}
\usepackage{url}
\usepackage{verbatim}
\usepackage{graphicx}
\usepackage{cite}
\usepackage{bm}
\usepackage{bbm}
\usepackage[mathscr]{euscript}
\usepackage{hyperref}
\usepackage{subcaption}
\usepackage{lipsum}
\usepackage{multicol}
\usepackage{tabularray}
\usepackage{booktabs}
\usepackage{amsthm}

\newcommand*\wb{ \boldsymbol{w} }
\newcommand*\gb{ \boldsymbol{g} }
\newcommand*\hb{ \boldsymbol{h} }
\newcommand*\vb{ \boldsymbol{v} }
\newcommand*\sbb{ \boldsymbol{s} }
\newcommand*\w{ {\boldsymbol{\scriptstyle \mathcal{W} }} }
\newcommand*\y{ {\boldsymbol{\scriptstyle \mathcal{Y}} } }
\newcommand*\z{ {\boldsymbol{\scriptstyle \mathcal{Z} }} }
\newcommand*\s{ {\boldsymbol{\scriptstyle \mathcal{S} }} }
\newcommand*\g{ {\boldsymbol{\mathscr{G} }} }
\newcommand*\A{\mathcal{A}}
\newcommand*\J{\mathcal{J}}
\newcommand*\gradcal{\nabla\J}
\newcommand*\gradappcal{\widehat{\gradcal}}
\newcommand*\one{\mathbbm{1}}
\newcommand*\transp{^\mathsf{T}}
\newcommand*\scallones{\frac{1}{K}\one\one^{\mathsf{T}}}
\newcommand*\onekronim{\frac{1}{K}(\one\transp\otimes I_M)}
\newcommand*\Ihat{\widehat{\mathcal{I}}}
\newcommand*\what{\widehat{\w}}
\newcommand*\zhat{\widehat{\z}}
\newcommand*\Ahat{\widehat{\A}}

\newcommand*\xhat{{\scriptstyle \widehat{\boldsymbol{\mathcal{X}}}}}
\newcommand*\epstau{\epsilon_\tau}
\newcommand*\mathlinebreak{\nonumber\\ &}

\newcommand*\centerr{\widetilde{\wb}}

\newcommand\expect[1]{%
\mathbb{E}\left[#1\right]
}
\newcommand\expectfi[1]{%
\expect{#1 \Bigl| \w_{i-1}}
}
\newcommand\expectnormfi[1]{%
\expect{ \left\lVert#1\right\rVert^2 \Bigl|\w_{i-1}}
}
\newcommand\expectnorm[1]{%
\mathbb{E}\left\lVert#1\right\rVert^2
}
\newcommand\norm[1]{%
\left\lVert#1\right\rVert
}

\allowdisplaybreaks

\newtheorem{assumption}{Assumption}

\newtheorem{lemma}{Lemma}
\newtheorem{theorem}{Theorem}

\usepackage{color,soul}
\usepackage{mathtools}

\makeatletter
\renewcommand*\env@matrix[1][\arraystretch]{%
  \edef\arraystretch{#1}%
  \hskip -\arraycolsep
  \let\@ifnextchar\new@ifnextchar
  \array{*\c@MaxMatrixCols c}}
\makeatother

\expandafter\def\expandafter\normalsize\expandafter{%
    \normalsize%
    \setlength\abovedisplayskip{5pt}%
    \setlength\belowdisplayskip{6pt}%
    \setlength\abovedisplayshortskip{-8pt}%
    \setlength\belowdisplayshortskip{2pt}%
}

\begin{document}

\title{On the Convergence of Decentralized Stochastic Gradient-Tracking with Finite-Time Consensus}

\author{Aaron Fainman and Stefan Vlaski
        % <-this % stops a space
\thanks{The authors are with the Department of Electrical and Electronic Engineering, Imperial College London. This work was supported in part by EPSRC Grants EP/X04047X/1 and EP/Y037243/1. Emails: \{aaron.fainman22, s.vlaski\}@imperial.ac.uk.\\
A preliminary version of this work appeared in~\cite{asilomar}.\\
\textcolor{black}{Code associated with this work can be found in \url{https://github.com/aaronfainman/Approximate-FTC}}}
}

\maketitle

\begin{abstract}
Algorithms for decentralized optimization and learning rely on local optimization steps coupled with combination steps over a graph. Recent works have demonstrated that using a time-varying sequence of matrices that achieves finite-time consensus can improve the communication and iteration complexity of decentralized optimization algorithms based on gradient tracking. In practice, a sequence of matrices satisfying the exact finite-time consensus property may not be available due to imperfect knowledge of the network topology, a limit on the length of the sequence, or numerical instabilities. In this work, we quantify the impact of \emph{approximate} finite-time consensus sequences on the convergence of a gradient-tracking based decentralized optimization algorithm. Our results hold for any periodic sequence of combination matrices. We clarify the interplay between approximation error of the finite-time consensus sequence and the length of the sequence as well as typical problem parameters such as smoothness and gradient noise.
\end{abstract}


\section{Introduction}
We consider a network of $K$ agents aiming to collaboratively solve an aggregate optimization problem of the form:
\begin{align}\label{eq:cons-opt}
w^o\triangleq \underset{w\in\mathbb{R}^M}{\arg\min} ~J(w)=\underset{w\in\mathbb{R}^M}{\arg\min} ~\frac{1}{K}\sum_{k=1}^K J_k(w),
\end{align}
Here, each agent $k$ has access only to realizations of its local data \( \boldsymbol{x}_k \), giving rise to the local objectives \({ J_k(w) \triangleq \mathbb{E} Q_k(w; \boldsymbol{x}_k) } \). We represent the network through an undirected graph \( \mathcal{G} = (\mathcal{N}, \mathcal{E})\), where \( \mathcal{N} \) denotes the set of \( K \) agents, and \( \mathcal{E} \) denotes the set of edges. We denote by \( \mathcal{N}_k \) the neighborhood of agent \( k \), i.e., the set of agents with whom agent \( k \) shares an edge.

Decentralized strategies to solve problem~\eqref{eq:cons-opt} (see, for example, \cite{sayed2014, cattivelli, ram, extra, next, exact_diff, lian, vlaski_kar_23,pu21,tang18}) rely only on exchanges of information between pairs of neighbors and avoid any centralized processing. They generally consist of two main steps: a local update step, where agents update their model parameters using the gradient or some gradient approximation, and a mixing step, where agents collect and average intermediate quantities (such as models or gradients) with their neighbors via a consensus mechanism:
\begin{align}
    % w_{k,i} = \sum_{\ell \in \mathcal{N}_k } a_{lk,i} w_{i-1}
    \mathrm{Avg}\left( \{ u_k \}_{k=1}^K \right) = \sum_{\ell \in \mathcal{N}_k} a_{\ell k,i} u_{\ell}
\end{align}
Here, $a_{\ell k,i}=[A_i]_{\ell, k}$ are the elements of the combination matrix $A_i$ at time $i$.

Performance bounds for iterative decentralized optimization algorithms generally consist of an optimization error and a consensus error. The optimization error arises from the use of iterative gradient-based updates, is common to centralized, federated and decentralized strategies, and represents a lower bound on the performance of distributed optimization algorithms. The consensus error refers to the variation of local models maintained by individual agents, and arises from the limited diffusion of information over the network. It can become negligible for densely or even fully-connected topologies, but for sparsely connected networks, it can contribute significantly to the overall error. Efforts to mitigate the consensus error have included finding the optimal static weights for specific graph topologies~\cite{xiao}, restricting the optimization to specific graph topologies~\cite{song_comm-eff, ying-exp}, and using double loop algorithms to enforce consensus~\cite{jakovetic, bastianello23}.

An alternative approach is the use of finite-time consensus (FTC) sequences~\cite{Cesar, kibangou, safavi, sandryhaila}. A sequence of \( \tau \) matrices satisfies the finite-time consensus property, if their product equals the scaled all-ones matrix:
\begin{align}\label{eqn:ftc}
A_\tau \cdots A_2 A_1 = \scallones
\end{align}
FTC sequences arise from the study of linear averaging as a method to  achieve consensus amongst agents~\cite{ko, Ko_Shi_2009, hendrickx}. Classical weighting rules, such as the Metropolis-Hastings weights~\cite{sayed2014}, can only guarantee asymptotic convergence, while FTC sequences achieve exact consensus in a finite number of steps. They effectively mimic a fully connected network over $\tau$ steps, where $\tau$ is the consensus number for the graph. The graph diameter serves as a lower bound for $\tau$, while twice the graph's radius serves as its upper bound~\cite{hendrickx}.

The advantages of FTC sequences extend beyond the consensus problem and are applicable to more general optimization problems over graphs. FTC sequences have been utilized as the combination weights in decentralized algorithms with demonstrable benefits~\cite{Cesar, ying-exp, eusipco}. For instance, FTC sequences have enabled sparser communication between agents without compromising performance~\cite{Cesar, ying-exp}, while in another work, faster convergence rates were observed for gradient-tracking algorithms using FTC sequences~\cite{eusipco}.

In practical settings, an exact sequence satisfying~\eqref{eqn:ftc} may not be available, and instead \emph{approximate} FTC sequences can be used. Numerical methods for constructing these sequences, such as eigenvalue decompositions~\cite{kibangou, safavi,sandryhaila}, graph filtering techniques~\cite{coutino,segarra}, or decentralized learning approaches~\cite{eusipco}, may introduce inaccuracies. As a result, the generated sequences may only approximate the scaled all-ones matrix. We refer to these as \emph{approximate} FTC sequences, with the quality of the approximation characterized by $\epstau$:
\begin{align}\label{eq:approx-ftc}
    \epstau \triangleq \norm{A_\tau\cdots A_2 A_1 - \scallones}
\end{align}
Empirical evidence suggests that approximate FTC sequences can nevertheless provide substantial benefit to decentralized optimization~\cite{eusipco}. \textcolor{black}{We demonstrate this in Fig.~\ref{fig:approxFTC-benefit} for a decentralized linear least squares problem, with simulation details provided in Sec.~\ref{sec:discussion}. For the considered graph, no analytical FTC sequence is known. An approximate sequence can be obtained using the numerical method in~\cite{eusipco} ($\epstau=0.7$,$\tau=4$), which we compare with the fastest distributed linear averaging (FDLA) matrix~\cite{xiao}, a static combination matrix designed to have the largest mixing rate between agents. Gradient-tracking performance with the approximate sequence clearly outperforms the FDLA case, demonstrating that approximate FTC sequences can improve performance despite not achieving exact consensus.} 

\textcolor{black}{Since exact FTC sequences are only known for a handful of graph families, and approximate sequences demonstrate empirical benefits, these observations motivate a deeper analysis into their behavior. } In this work, we analytically quantify the impact of inaccuracies in FTC sequences when used in gradient-tracking schemes, \textcolor{black}{highlighting both the benefits and trade-offs of their use}.

\begin{table*}[]
\begin{center}
\caption{Performance bounds of decentralized strategies with and without (approximate) finite-time consensus}
\begin{tabular}{@{}lcccc@{}}
\toprule
\multicolumn{1}{c}{Method} & \multicolumn{1}{c}{Objective Function} & \multicolumn{1}{c}{Combination Matrix} & \multicolumn{1}{c}{Steady-State Performance} & \multicolumn{1}{c}{Rate of Convergence} \\ \midrule
ATC-Diffusion\cite{yuan20} & Strongly Convex & Static & $\mathcal{O}\left( \frac{\mu\sigma^2}{ K} +\frac{\mu^2\lambda^2\sigma^2}{1-\lambda}+\frac{\mu^2\lambda^2\zeta^2}{(1-\lambda)^2}\right)$  & $1-\Theta(\mu\nu)$ \\
ATC-GT\cite{alghunaim}  & PL-Condition & Static & $\mathcal{O}\left(\frac{\mu\sigma^2}{K} + \frac{\mu^2\lambda^4\sigma^2}{1-\lambda}+\frac{\mu^4\lambda^4\sigma^2}{K(1-\lambda)^4} \right)$ & $1-\Theta(\mu\nu)$ \\
DIGing\cite{nedic17} & Strongly Convex & Time-Varying & - & $\left(1-\Theta(\mu\nu)\right)^{\frac{1}{2\tau}}$\\
NEXT with FTC\cite{Cesar} & Non Convex & FTC & $\mathcal{O}\left(\frac{\mu\sigma^2}{K} +\mu^2\tau^3\sigma^2 +\frac{\mu^4\tau^4\sigma^2}{K}\right)$ & - \\
Aug-DGM with FTC (this work) & Strongly Convex & Approximate FTC & $\mathcal{O}\left(\frac{\mu\sigma^2}{K} +\frac{\mu^2\tau^2\sigma^2}{K\left(1-\epstau\right)^2} +\frac{\mu^2\tau^2\sigma^2}{(1-\epstau)^2}\right)$ & $1-\Theta(\nu\mu)+\Theta(\nu\epstau\mu)$\\
% Aug-DGM with FTC & This work & Strongly Convex & Approximate FTC & $\mathcal{O}\left(\frac{\mu\sigma^2}{K} +\frac{\mu\zeta^2}{K} +\frac{\mu^2\tau^2(\sigma^2+\zeta^2)}{K\left(1-\epstau\right)^2} +\frac{\mu^2\tau^2(\sigma^2+\zeta^2)}{(1-\epstau)^2}\right)$\\
\bottomrule
\label{tab:performance}\end{tabular}\\
where $\lambda$ is the second-largest eigenvalue for the static combination matrix and where we have assumed $\beta_k$ in~\eqref{eq:grad_k-var} is 0 so\\ \hspace{-220pt}that the noise terms in this work match those of~\cite{Cesar}.
\end{center}
\vspace{-15pt}
\end{table*}
\textcolor{black}{ Any periodic sequence of combination matrices can be considered an approximate FTC sequence and thus our results hold generally for gradient-tracking with time-varying but periodic combination matrices. The same recursion could also be studied in the context of general time-varying matrices (see, for example,~\cite{next,nedic16,nedic17,assran19,scutari19,saadatniaki20,huan24}). However, leveraging the periodic structure of the matrices allows us to derive a tighter bound. For example, under the assumption of $\tau$-connectivity (see Assumption~\ref{assump:comb-mat}), linear convergence in~\cite{ nedic17} is proven with a rate of convergence of~$\left(1-\Theta(\mu\nu)\right)^{\frac{1}{2\tau}}$. On the other hand, we prove convergence with an amortized rate of $1-\Theta(\nu\mu)+\Theta(\nu\mu\epstau)$\textcolor{black}{, which is faster by a factor of \( \tau\)}. Our results are summarized in Table~\ref{tab:performance}.}

\textcolor{black}{ We also highlight an interesting phenomenon in the dynamics of the recursion, namely that of a two timescale behavior. The error displays a general downwards trend in periods of $\tau$, however, iterate-on-iterate, the error may grow and shrink, as shown in Fig.s~\ref{fig:numberline} and~\ref{fig:tradeoff}(b). This non-monotonic behavior introduces challenges in the analysis. We therefore derive a bound in terms of the maximum error in periods of $\tau$, in contrast to the usual iterate-on-iterate performance guarantees.}

\textcolor{black}{\noindent We summarize our contributions as follows:
\begin{itemize}
    \item We derive performance guarantees for a gradient-tracking algorithm with approximate FTC sequences. Our bound holds for any periodic sequence of matrices, and is substantially improved compared to existing bounds in the literature that do not exploit the periodic nature of FTC sequences. Results are summarized in Table~\ref{tab:performance}.
    \item We provide a detailed analysis which is able to reconcile the two-timescale behavior of the algorithm: monotonic decrease of the centroid error and per-\( \tau \) iterate decrease for the consensus error. This technique enables us to capture the effect of the finite-time consensus mechanism on every single iteration, and results in tighter bounds for convergence rate and steady-state performance.
    \item  We characterize the rate at which performance improves as the approximation error decreases (i.e., as $\epstau$ decreases) and reveal that FTC sequences are particularly beneficial when the consensus number, $\tau$, is relatively low. The derived tradeoffs between $\tau$ and $\epstau$ indicate that for certain network topologies, approximate FTC sequences yield more benefit than exact FTC sequences. 
    % \item We corroborate our theoretical results with numerical experiments and provide the FTC sequences for path and ring networks.
\end{itemize}}

\begin{figure}
\centering

\begin{minipage}{0.3\linewidth}
    \centering
    \includegraphics[width=\linewidth]{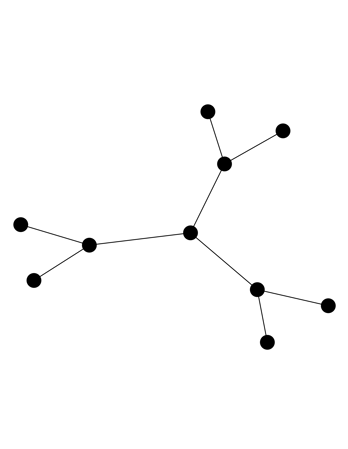}
    \vspace{10pt}\caption*{(a) Graph}
\end{minipage}
\hspace{5pt}
\begin{minipage}{0.62\linewidth}
    \centering
    \includegraphics[width=\linewidth]{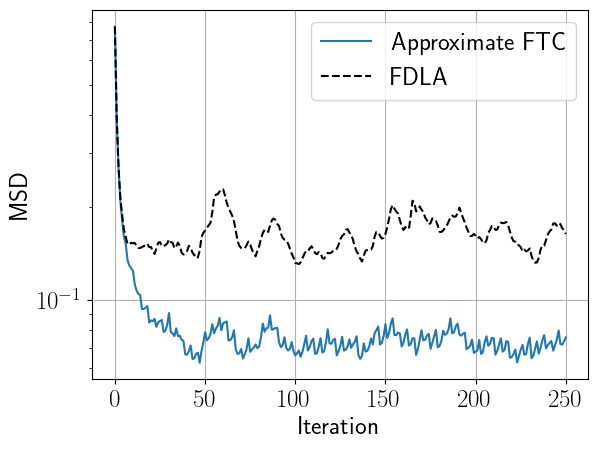}
    \caption*{(b) Performance}
\end{minipage}
\caption{\textcolor{black}{The performance of gradient-tracking can be improved using approximate FTC sequences. Here we compare with the fastest distributed linear averaging (FDLA) matrix~\cite{xiao}.}}
\label{fig:approxFTC-benefit}
\end{figure}

\vspace{-10pt}\subsection{Notation}
We use uppercase letters to denote matrices, lowercase letters for vectors, and Greek letters for scalars. Boldface symbols represent random quantities. The vector of all ones of size $K$ is denoted by $\one_K$ while $I_M$ is the $M\times M$ identity matrix. The modulo operation is denoted by $\%$. The colon subscript is used as a shorthand for the product of matrices, i.e. ${A_{m:n}=A_mA_{m-1}\cdots A_n}$. For vectors, $\norm{\cdot}$ refers to the Euclidean norm and for matrices it denotes the spectral norm.

\section{Analysis}
In principle, FTC sequences can be applied in any decentralized algorithm. In this work, we analyze a variant of the gradient-tracking algorithm, Aug-DGM~\cite{augdgm}. At each iteration $i$, agent $k$ updates the model variable, $\wb_{k,i}$, and an auxiliary variable that tracks the gradient of the aggregate cost, $\gb_{k,i}$, according to:
\begin{subequations}\label{eqn:aug-dgm-node}
\begin{align}
    \wb_{k,i} &= \sum_{\ell\in\mathcal{N}_k} a_{\ell k,i}\left(\wb_{\ell,i-1}-\gb_{\ell,i-1} \right) \label{eqn:aug-dgm-node:w} \\
    \gb_{k,i} &=  \sum_{\ell\in\mathcal{N}_k} a_{\ell k,i}\left( \gb_{\ell,i-1} + \mu\widehat{\nabla J}_\ell(\wb_{\ell,i})-\mu\widehat{\nabla J}_\ell(\wb_{\ell,i-1})\right) \label{eqn:aug-dgm-node:g}
\end{align}
\end{subequations}
where $\mu$ is the step size and the combination matrices are cycled over the FTC sequence, such that, ${A_i=A_{i\%\tau}}$. The stochastic approximation of the true gradient, $\nabla J_k(\cdot)$, is represented by $\widehat{\nabla J_k}(\cdot)$. For example, we can construct ${\widehat{\nabla J}(\wb_{k,i-1})\triangleq \nabla Q_k(\wb_{k,i-1}; \boldsymbol{x}_{k,i})}$. \textcolor{black}{ The algorithm is initialized with ${\wb_{k,0}\in\mathbb{R}^M}$ and ${\gb_{k,0}=\mu\widehat{\nabla J}_k(\wb_{k,0})}$}.

Observe that the FTC sequence is interlaced with the ordinary running of the algorithm in \eqref{eqn:aug-dgm-node:w}--\eqref{eqn:aug-dgm-node:g}. The benefit of this construction is that \eqref{eqn:aug-dgm-node:w}--\eqref{eqn:aug-dgm-node:g} is still a single-timescale algorithm, with no increase in its per-iteration communication complexity. The challenge in analyzing the dynamics of this set of recursions is that now, unlike in the case of classical consensus, agents will continue to update their local models along their local gradients over the course of one FTC sequence of length \( \tau \), which will prevent exact consensus after \( \tau \) steps, even when \( \epsilon_{\tau} \) is small or even vanishes. We will present a novel analysis framework which allows for the precise quantification of these drift effects.

As in~\cite{Cesar, asilomar}, we deviate in our formulation of Aug-DGM by multiplying the gradient approximations by $\mu$ in~\eqref{eqn:aug-dgm-node:g} rather than $\gb_{\ell,i-1}$ in~\eqref{eqn:aug-dgm-node:w} as in the original variant of the algorithm~\cite{augdgm}. This helps limit the effect of the drift between agents, as discussed in Section~\ref{sec:discussion}.

We can write the recursion more compactly by defining network-level quantities, ${\w_{i}\triangleq\textrm{col}\{\wb_{k,i} \}}$, ${\g_{i}\triangleq\textrm{col}\{\gb_{k,i} \}}$, ${\gradappcal(\w_i)=\textrm{col}\{ \widehat{\nabla J}(\wb_{k,i}) \}}$ and combination weights ${\A_i=A_i\otimes I_M}$ so that the coupled recursion becomes:
\begin{subequations}\label{eq:aug-dgm}
    \begin{align}
        \w_i &=\A_i\left( \w_{i-1}-\g_{i-1} \right)\label{eq:aug-dgm:w} \\
        \g_i &= \A_i\left(\g_{i-1}+\mu\gradappcal(\w_i)-\mu\gradappcal(\w_{i-1}) \right)\label{eq:aug-dgm:g} 
    \end{align}
\end{subequations}

We define the network centroid at iteration $i$ as ${\wb_{c,i}\triangleq \frac{1}{K}(\one\transp\otimes I_M)\w_i=\frac{1}{K}\sum_{k=1}^K \wb_{k,i} }$ and its stacked version as ${\w_{c,i}\triangleq \one\otimes\wb_{c,i}}$. For the convenience of analysis we transform the coupled recursion in~\eqref{eq:aug-dgm}. By defining the change of variable ${\y_i\triangleq\g_i-\mu\A_i\gradappcal(\w_i)}$ and then \textcolor{black}{${\z_i\triangleq\y_i+\mu \A_i\gradcal(\w_{c,i})}$}, the equation pair becomes:
\begin{subequations}\label{eqn:aug-dgm-z}
    % \begin{align}
    \begin{align}
        \w_i =&\: \A_i\w_{i-1}-\A_i\z_{i-1}\nonumber\\&+\mu\A_i\A_{i-1}\Big(\gradcal(\w_{c,i-1})
        -\gradcal(\w_{i-1})\Big) \nonumber\\&-\mu\A_i\A_{i-1}\s_i(\w_{i-1})\label{eqn:aug-dgm-z:w}\\
        \z_i =&\: \A_i\z_{i-1} -\mu\A_i(I-\A_{i-1})\s_i(\w_{i-1})\nonumber\\&-\mu\A_i(I-\A_{i-1})(\gradcal(\w_{i-1})-\gradcal(\w_{c,i-1}))\nonumber\\& + \mu\A_i(\gradcal(\w_{c,i})-\gradcal(\w_{c,i-1}))
        \label{eqn:aug-dgm-z:z}\end{align}
\end{subequations}
where we define the gradient noise as ${\sbb_{k,i}(\wb_{i-1})\triangleq \widehat{\nabla J}_k(\wb_{k,i-1})-\nabla J_k(\wb_{k,i-1})}$ and its stacked version as $\s_i(\w_{i-1})\triangleq \text{col}\{\sbb_{k,i}(\wb_{k,i-1}) \}$. Details of the transformation can be found in Appendix~\ref{app:transformation}.

We begin the analysis by separately considering the centroid error of the model, defined as ${\centerr_{c,i}\triangleq w^o-\wb_{c,i}}$ for the optimal model $w^o$, and the disagreement between agents' local models ${\what_i\triangleq\Ihat\w_i=\w_i-\one\otimes\wb_{c,i}}$ where ${\Ihat\triangleq(I_K-\scallones)\otimes I_M}$. We will consider the disagreement jointly for both $\w_i$ and $\z_i$ (defined as ${\zhat_i\triangleq\Ihat\z_i})$ which we term the consensus error, ${\xhat_i\triangleq[\what_i\transp,\zhat_i\transp]\transp}$.

\vspace{-5pt}\subsection{Proof Sketch}
\textcolor{black}{The analysis proceeds by separately considering the behavior of the consensus error, $\xhat_i$, and the centroid error $\centerr_{c,i}$. The behavior of the two follow different timescales, as shown in Figure~\ref{fig:numberline}.}

\textcolor{black}{The evolution of the centroid error can be bounded from one iteration to the next, as is typically the case in gradient-based optimization algorithms. The consensus error, on the other hand, is more suitably bounded over a full consensus sequence of length \( \tau \), since the (approximate) consensus guarantee~\eqref{eq:approx-ftc} holds only at the end of a full FTC sequence.} 

\textcolor{black}{In Section~\ref{sec:consensus-err} we describe the dynamics of the consensus error over this longer timescale, deriving a bound in Lemma~\eqref{lemma:consensus-bound} for $\xhat_i$ in terms of the errors from the previous iteration $i-1$ back to the most recent complete FTC sequence, ${(\left\lfloor\frac{i}{\tau}\right\rfloor-1)\tau}$. }

\textcolor{black}{In Section~\ref{sec:centroid-err} we bound the centroid error $\centerr_{c,i}$ in relation to $\centerr_{c,i-1}$ and $\xhat_{i-1}$ (see Lemma~\eqref{lemma:centroid_bound}).}

\textcolor{black}{Reconciling the coupled recursions from Lemmas~\eqref{lemma:consensus-bound} and~\eqref{lemma:centroid_bound} over two separate timescales is the core challenge in the convergence analysis, and requires novel techniques. In Section~\ref{sec:main-res} we bridge these two timescales by considering the maximum error over the sequence, deriving the coupled recursion for the maximum centroid and consensus errors from one sequence to the next. Our main convergence result follows in Theorem~\eqref{thm:main-res}.}

\begin{center}
\begin{figure}[]
    \centering % <-- added
  \includegraphics[width=\linewidth]{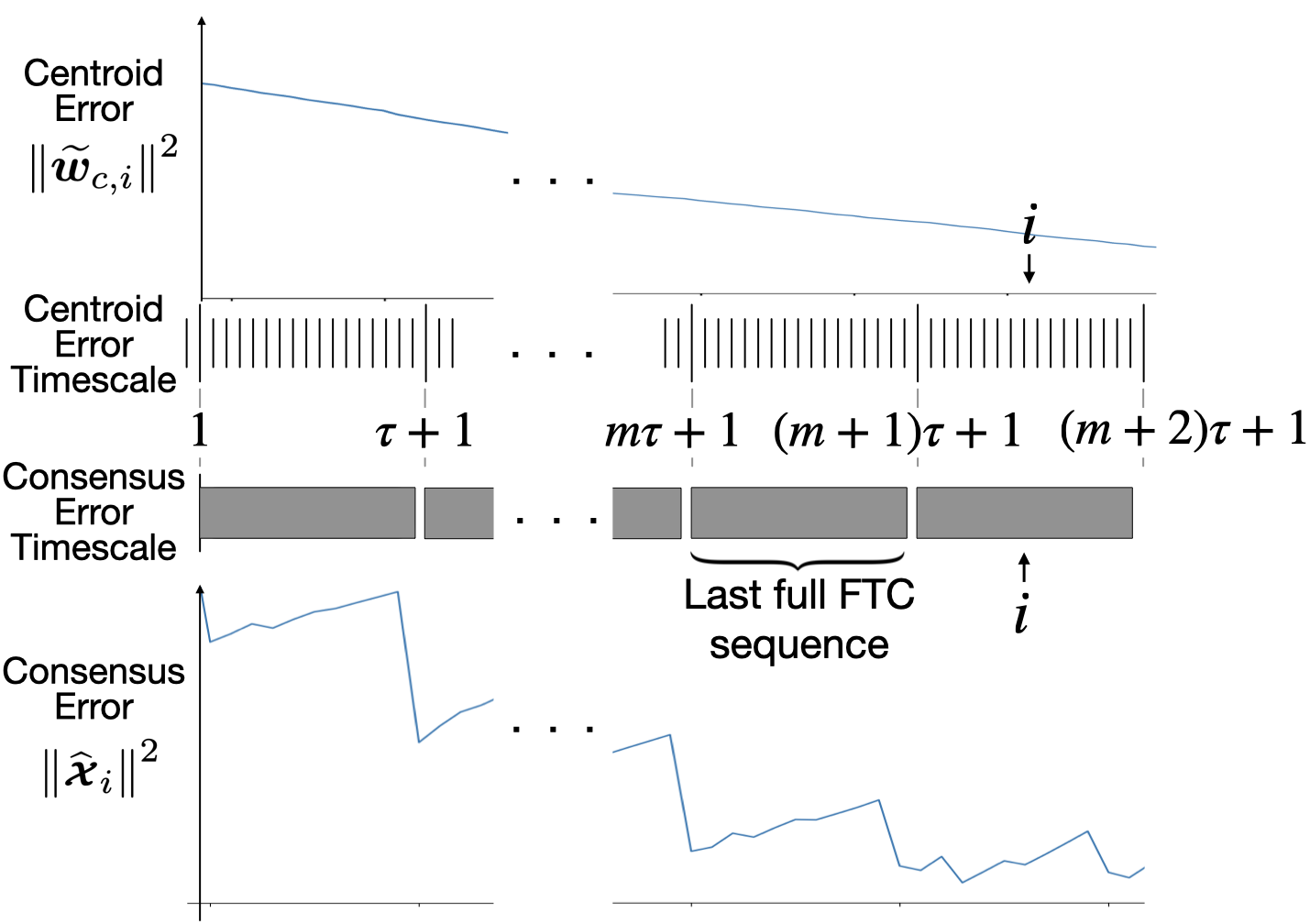}
\caption{The two-timescale behavior of the algorithm depicted with real data. The consensus error $\xhat_i$ decays over sequences of length $\tau$. It is bound in Sec.~\ref{sec:consensus-err}. The centroid error, $\centerr_{c,i}$, depends on the previous iterate alone and is bound in Sec.~\ref{sec:centroid-err}. The two errors are combined into the main result in Sec.~\ref{sec:main-res}.}
\label{fig:numberline}
\end{figure}
\end{center}

\vspace{-35pt}
\subsection{Preliminaries}\label{sec:assumptions}
We conduct the analysis under the following common regularity conditions on the objective functions, gradient noise and combination matrices.
\begin{assumption}[Regularity conditions]\label{assump:regularity}
The aggregate objective function $J(\cdot)$ is \( \nu\)-strongly convex, so that for all \( x, y \in \mathbb{R}^M: \)
\begin{align}
    J(y)\geq J(x)+\nabla J(x)\transp(y-x)+\frac{\nu}{2}\lVert y-x\rVert^2 
\end{align}
and the local gradients are Lipschitz smooth:
\begin{align}\label{eq:assump-lipsch}
    \lVert \nabla J_k(x)-\nabla J_k(y)\rVert \leq \delta\lVert x-y\rVert
\end{align}
\end{assumption}

\begin{assumption}[Gradient noise]\label{assump:gradnoise}
The gradient noise, $\sbb_{k,i}(\wb_{k,i-1})$ is unbiased, pairwise-uncorrelated and has a bounded variance conditioned on $\w_{i-1}$:
\begin{subequations}
    \begin{align}
        &\expectfi{\sbb_{k,i}(\wb_{k,i-1})} =0 \\
        &\expectfi{\sbb_{k,i}\transp(\wb_{k,i-1}) \sbb_{\ell,i}(\wb_{\ell,i-1})} =0 \label{eq:assump-gradnoise-unbias}\\
        &\expectfi{\left\lVert \boldsymbol{s}_{k,i}(\boldsymbol{w}_{k,i-1}) \right\rVert^2}\leq \beta_k^2\lVert w_k^o-\boldsymbol{w}_{k,i-1}\rVert^2+\sigma_k^2  \label{eq:grad_k-var}
    \end{align}\vspace{-2pt}
    where $w_k^o$ is the minimizer of $J_k(w)$ \textcolor{black}{and $\beta_k^2$, $\sigma_k^2$ are the local gradient noise parameters}.

Under Assumption~\eqref{assump:gradnoise}, it can be readily verified that the stacked gradient noise vector $\s_i(\w_{i-1})$ is thus also conditionally unbiased. The variance is bounded by:
    \begin{align}
        \expectnormfi{\s_i(\w_{i-1})} &\leq 3\beta^2\zeta^2+3\beta^2K\norm{\centerr_{c,i-1}}^2\mathlinebreak+3\beta^2\norm{\what_{i-1}}^2+\sigma^2 \label{eq:stacked-gradnoise-var}
    \end{align}
where we have introduced the local heterogeneity term as ${\zeta_k^2\triangleq\norm{w_k^o-w^o}^2}$, the aggregate heterogeneity as ${\zeta^2\triangleq\sum\limits_{k=1}^K\zeta_k^2}$, and the aggregate gradient noise parameters ${\sigma^2\triangleq\sum_{k=1}^K\sigma_k^2}$, ${\beta^2\triangleq\sum_{k=1}^K\beta_k^2}$.

Similarly, the centroid of the stacked gradient noise vector is bounded by:
\begin{align}
    &\expectnormfi{\onekronim\s_i(\w_{i-1})} \leq \mathlinebreak\qquad \frac{3\beta^2\zeta^2}{K}+3\beta^2\norm{\centerr_{c,i-1}}^2+\frac{3\beta^2}{K}\norm{\what_{i-1}}^2+\frac{\sigma^2}{K} \label{eq:avg-stacked-gradnoise-var}
\end{align}
\end{subequations}
\end{assumption}
\begin{assumption}[Combination Matrices]\label{assump:comb-mat}
Each combination matrix ${\{A_j\}_{j=1}^{\tau}}$ in the FTC sequence is \textcolor{black}{symmetric}, doubly-stochastic and has a spectral radius of one. We do not assume that the individual combination matrices are primitive or even connected, \textcolor{black}{however, the graph associated with the product of matrices in the sequence is strongly-connected ($\tau-$strongly-connected)}.
\end{assumption}

\begin{lemma}[Approximation Error Bound]
    \textcolor{black}{Under assumption~\eqref{assump:comb-mat} it holds that $\epstau<1$.}
\end{lemma}
\begin{proof}
\textcolor{black}{The product of matrices in the sequence, $A_{\tau:1}$, is doubly-stochastic: $A_{\tau:1}\one_K=\one_K$. Consider now, $\epstau$:}
\begin{align}
&\epstau^2\triangleq \norm{A_{\tau:1} - \scallones}_2^2 \nonumber\\
&= \rho\left( \left(A_{\tau:1} - \scallones\right)\transp\left(A_{\tau:1} - \scallones\right)\right) \nonumber \\
&=\rho\left(A_{\tau:1}\transp A_{\tau:1} - \scallones\right)
\end{align}

\noindent \textcolor{black}{$A_{\tau:1}\transp A_{\tau:1}$ is symmetric and doubly-stochastic. Strong-connectivity implies that $A_{\tau:1}$ (and thus $A_{\tau:1}\transp A_{\tau:1}$) is primitive. From the Perron-Frobenius theorem, it follows then that 1 is a simple eigenvalue of $A_{\tau:1}\transp A_{\tau:1}$ and that $\rho(A_{\tau:1}\transp A_{\tau:1})=1$. Hence, $ {\rho\left(A_{\tau:1}\transp A_{\tau:1} - \scallones\right)<1}$ and thus $\epstau^2<1$.\\}
\end{proof}

\vspace{-15pt}\subsection{Consensus Error}\label{sec:consensus-err}
We can compactly describe the dynamics of the consensus error as~\cite{Cesar}:
\begin{subequations}
\begin{equation}\label{eqn:consensus-err-evolution}
    \xhat_i = G_i\xhat_{i-1}-\mu\hb_i-\mu\vb_i
\end{equation}
    where:
    \begin{align}
        &\xhat_i\triangleq \begin{bmatrix}
            \what_i\\
            \zhat_i
        \end{bmatrix}\\
        &\Ahat_i\triangleq \Ihat\A_i \\
&G_{i} \triangleq  \begin{bmatrix}
\Ahat_{i} & - \Ahat_{i} \\
0 & \Ahat_{i} 
\end{bmatrix}  \\
&\boldsymbol{h}_{i} \triangleq \begin{bmatrix*}[l]
\Ahat_{i}\Ahat_{i-1}(\gradcal(\w_{i-1})-\gradcal(\w_{c,i-1})) \\
\A_{i} (I-A_{i-1})(\gradcal(\w_{i-1})-\gradcal(\w_{c,i-1}))-\\ \hspace{80pt}\Ahat_{i}(\gradcal(\w_{c,i})-\gradcal(\w_{c,i-1}))
\end{bmatrix*} \label{eqn:h-defn}\\
&\boldsymbol{v}_{i} \triangleq \begin{bmatrix}
    \Ahat_{i}\Ahat_{i-1}\s_{i}(\w_{i-1}) \\
    \Ahat_{i}(I-\A_{i-1})\s_{i}(\w_{i-1})
\end{bmatrix}\label{eqn:v-defn}
    \end{align}
\end{subequations}
which follows after premultiplying~\eqref{eqn:aug-dgm-z} by $\Ihat$ and noting that ${\Ihat\A_i=\Ahat_i=\Ahat_i\Ihat}$ and ${\Ihat(I-\A_i)=I-\A_i}$.

In order to make use of the FTC property in~\eqref{eq:approx-ftc} we must consider at least a full sequence of combination matrices. We do this by iterating backwards over the consensus error in~\eqref{eqn:consensus-err-evolution} from time $i$ until the start of the last full FTC sequence, ${m\tau+1}$ where ${m=\left\lfloor\frac{i}{\tau}\right\rfloor-1}$ as shown in Figure~\ref{fig:numberline}. We begin by substituting~\eqref{eqn:consensus-err-evolution} into $\xhat_{i-1}$:
\begin{align}
    \xhat_{i} &= G_{i:i-1}\xhat_{i-2}-\mu G_i\hb_{i-1} - \mu \hb_i -\mu G_i\vb_{i-1} - \mu \vb_i
    \end{align}
    and then substitute~\eqref{eqn:consensus-err-evolution} again in $\xhat_{i-3}$
    \begin{align}
     \xhat_{i} &= G_{i:i-2}\xhat_{i-3}-\mu G_{i:i-1}\hb_{i-2}-\mu G_i\hb_{i-1}- \mu \hb_i \mathlinebreak\qquad-\mu G_{i:i-1}\vb_{i-2}-\mu G_i\vb_{i-1}- \mu \vb_i 
     \end{align}
     Repeating the process until time $m\tau+1$:
     \begin{align}
    \xhat_i &= G_{i:m\tau+1}\xhat_{m\tau}-\mu\sum_{\mathclap{j=m\tau+1}}^{i-1}G_{i:j+1}\hb_{j}-\mu\hb_i\mathlinebreak\qquad-\mu\sum_{\mathclap{j=m\tau+1}}^{i-1}G_{i:j+1}\vb_{j} -\mu\vb_i \label{eqn:cons-err-mtau}
    \end{align}
Iterating backwards to time~$m\tau+1$ ensures that ${G_{i:m\tau+1}}$ consists of the product of between $\tau$ and $2\tau-1$ matrices, which will be necessary to employ the FTC property in~\eqref{eqn:ftc}.

In order to bound~\eqref{eqn:cons-err-mtau} we establish the following two lemmas, the proofs of which are found in Appendix~\ref{app:cons-err-relation}.

\begin{lemma}[Perturbation Bounds]\label{lemma:v-h-bound} Under Assumptions~\eqref{assump:gradnoise} and~\eqref{assump:comb-mat}, the vector $\vb_i$ is conditionally unbiased and bounded by:
    \begin{align}
        \expectfi{\lVert \boldsymbol{v}_i\rVert^2}&\leq 
        9\beta^2\zeta^2+9\beta^2K\norm{\centerr_{c,i-1}}^2\mathlinebreak \qquad +9\beta^2\norm{\what_{i-1}}^2+3\sigma^2\label{eq:vi-var}
    \end{align}
and, $\boldsymbol{h}_i$, is bounded by:
    \begin{align}\label{eq:h-bound}
        &\expectnormfi{\boldsymbol{h}_i}\leq 3\left(2\delta^2+\frac{1}{2}\beta^2\right)\norm{\what_{i-1}} \mathlinebreak\quad+ 2(\delta^2+\beta^2K)\norm{\centerr_{c,i-1}}^2+\frac{3}{2}\beta^2\zeta^2+\frac{1}{2}\sigma^2
    \end{align}
    for $\mu\leq\frac{1}{2\delta}$.
\end{lemma}
\begin{proof}
    See Appendix~\ref{app:cons-err-relation}.
\end{proof}

The characterization of the consensus error is presented below. It bounds the consensus error at time \( i \) in terms of the consensus error at the beginning of the last full consensus sequence as well as the consensus and centroid errors encountered throughout the sequence. The two separate timescales of the bound introduces challenges in the downstream analysis, but is necessary to establish tight links between the consensus error of the FTC sequence and the distance of the centroid to the globally optimal model \( w^o \). In Sec.~\ref{sec:main-res} further below we show how these two timescales can be reconciled.
\begin{lemma}[Consensus Error]\label{lemma:consensus-bound}
    Under Assumptions~\eqref{assump:regularity}-\eqref{assump:comb-mat}, for ${0\leq\epstau<1}$ and $\mu\leq\frac{1}{2\delta}$, the consensus error is bounded by:
    \begin{subequations}
        \begin{align}\label{eq:lemma:consensus-bound}
            \expectnorm{\xhat_i}&\leq \theta_1\expectnorm{\xhat_{m\tau}} +\theta_2\sum_{\mathclap{j=m\tau}}^{i-1}\expectnorm{\xhat_j}\mathlinebreak\qquad+\theta_3\sum_{\mathclap{j=m\tau}}^{i-1}\expectnorm{\centerr_{c,j}}+\theta_4 
        \end{align}
    where :
    \begin{align}
    \    \theta_1 &\triangleq \frac{\epstau}{2}(1+\epstau)        \\
    \theta_2 &\triangleq    \frac{3\tau(4\delta^2+\beta^2)(1+\epstau)\mu^2}{1-\epstau}+18\tau\beta^2\mu^2    \\
    \theta_3 &\triangleq \frac{4\tau(\delta^2+\beta^2K)(1+\epstau)\mu^2}{1-\epstau} +18\tau\beta^2K\mu^2        \\
    \theta_4 &\triangleq 2\tau\left(\frac{3\tau\beta^2(1+\epstau)}{1-\epstau}+18\tau\beta^2\right)\mu^2\zeta^2 \mathlinebreak\quad+ 2\tau\left(6\tau+\frac{1+\epstau}{1-\epstau}\right)\mu^2\sigma^2 
    \end{align}
    \end{subequations}
\end{lemma}
\begin{proof}
    See Appendix~\ref{app:cons-err-relation}.
\end{proof} 

The first term in the consensus error bound in~\eqref{eq:lemma:consensus-bound} depends on the consensus error at the end of the last full FTC sequence. This term arises due to the use of approximate sequences, disappearing in the perfect FTC case ($\epstau=0$). The second and third terms in~\eqref{eq:lemma:consensus-bound} indicate that the errors accumulate over the course of the sequence. Longer sequences allow local models to drift further apart, which can have a detrimental effect on overall learning performance when \( \tau \) is very large. We further discuss this in Section~\ref{sec:discussion}.

\vspace{-8pt}\subsection{Centroid Error}\label{sec:centroid-err}
The behavior of the centroid is described by:
\begin{align}\label{eqn:centroid-evolution}
    \wb_{c,i}=\wb_{c,i-1}-\frac{\mu}{K}(\one\transp\otimes I_M)(\gradcal(\w_{i-1})+\s_{i}(\w_{i-1})) 
\end{align}
which follows by premultiplying~\eqref{eq:aug-dgm-y:w} by $\onekronim$, noting that $\onekronim\y_i=0$ (see Appendix~\ref{app:cent-err}). The behavior of the centroid error is bound below.

\begin{lemma}[Centroid Error]\label{lemma:centroid_bound}
    Under Assumptions~\eqref{assump:regularity}-\eqref{assump:comb-mat} for $\mu\leq\frac{\nu}{\delta^2}$, the centroid error $\centerr_{c,i}$ is bounded by:
\begin{subequations}
    \begin{align}
    \expectnorm{\centerr_{c,i}}&\leq\alpha_1\expectnorm{\centerr_{c,i-1}} + \alpha_2\expectnorm{\xhat_{i-1}}+\alpha_3
\end{align}
where:
\begin{align}
    \alpha_1 &\triangleq  \textcolor{black}{1-\frac{\nu\mu}{2}}  \\
    \alpha_2 &\triangleq  \frac{2\delta^2\mu}{\nu K}+\mathcal{O}(\mu^2) \\
    \alpha_3 &\triangleq  \frac{3\beta^2\mu^2}{K}\zeta^2+\frac{\mu^2}{K}\sigma^2
    \end{align}
    \end{subequations}
\end{lemma}
\begin{proof}
    \textcolor{black}{See Appendix~\ref{app:cent-err}}.
\end{proof}

\vspace{-20pt}\subsection{Main Result}\label{sec:main-res}
Below we combine the consensus error and centroid error bounds from Lemmas~\eqref{lemma:consensus-bound} and~\eqref{lemma:centroid_bound} to derive the main result. The key challenge arises from the structure of the recursive dependencies: the centroid error bound at iteration $i$ depends only on the error of the previous iteration, while the consensus error bound depends on the sum of errors across the full FTC sequence from time $m\tau$ to $i-1$. Previous gradient-tracking analyses with perfect FTC ($\epstau=0$) tackled this by considering the average consensus error terms from iteration zero to $i$~\cite{Cesar}. However, when approximate sequences are employed ($\epstau>0$), this series will not converge. Instead, we will consider the maximum consensus error within each FTC sequence, resulting in a bound characterizing the maximum error over the $\tau$ iterations of any sequence. 

Using the shorthand $\mathcal{S}_{m}$ to denote the {$m$-th} sequence, or the set of integers ${\{(m-1)\tau,\cdots, m\tau-1\}}$ and taking the maximum of both sides of~\eqref{eq:lemma:consensus-bound} over the FTC sequence ${i\in\mathcal{S}_{m+2}}$ is:
\begin{align}
    &\max_{i\in\mathcal{S}_{m+2}} \expectnorm{\xhat_i} \leq \max_{i\in\mathcal{S}_{m+2}}\Big( \theta_1\expectnorm{\xhat_{m\tau}} +\theta_2\sum_{\mathclap{j=m\tau}}^{i-1}\expectnorm{\xhat_j}\mathlinebreak\qquad+\theta_3\sum_{\mathclap{j=m\tau}}^{i-1}\expectnorm{\centerr_{c,j}}+\theta_4  \Big) \\
    &\stackrel{(a)}{=} \theta_1\expectnorm{\xhat_{m\tau}} +\theta_2\sum_{\mathclap{j=m\tau}}^{\mathclap{(m+2)\tau-2}}\expectnorm{\xhat_j}+\theta_3\sum_{\mathclap{j=m\tau}}^{\mathclap{(m+2)\tau-2}}\expectnorm{\centerr_{c,j}}+\theta_4 \\
    &\stackrel{(b)}{\leq} \theta_1\expectnorm{\xhat_{m\tau}} + \tau\theta_2\max_{~\mathclap{i\in\mathcal{S}_{m+1}}}\expectnorm{\xhat_i} + (\tau-1)\theta_2\max_{~\mathclap{i\in\mathcal{S}_{m+2}}}\expectnorm{\xhat_i}\nonumber\\&\quad+\tau\theta_3\max_{~\mathclap{i\in\mathcal{S}_{m+1}}}\expectnorm{\centerr_{c,i}} + (\tau-1)\theta_3\max_{~\mathclap{i\in\mathcal{S}_{m+2}}}\expectnorm{\centerr_{c,i}}+\theta_4
\end{align}
where $(a)$ follows by adjusting the summation limits to take the maximum value of $i$ over the indices, and in $(b)$ we split the summation into the ranges ${m\tau\leq j\leq (m+1)\tau-1}$ and ${(m+1)\tau\leq j\leq (m+2)\tau-1}$, replacing each with the maximum over the summation. 

For notational convenience, let the maximum centroid error over the $m$-th sequence be $\omega_{m}=\max_{i\in\mathcal{S}_{m}}\expectnorm{\centerr_{c,i}}$ and let ${\chi_{m}=\max_{i\in\mathcal{S}_{m}}\expectnorm{\xhat_i}}$ be the maximum consensus error over the {$m$-th} sequence so that the maximum consensus error bound becomes:% Using the fact that ${\expectnorm{\xhat_{m\tau}}\leq\max_{m\tau\leq i\leq (m+1)\tau-1}\expectnorm{\xhat_i}=\chi_{m+1}}$:
\begin{align} \label{eqn:max-cons-err}
    \chi_{m+2}&\leq \theta_1\chi_{m+1} + \tau\theta_2\chi_{m+1} + (\tau-1)\theta_2\chi_{m+2} + \tau\theta_3\omega_{m+1}\mathlinebreak\quad+(\tau-1)\theta_3\omega_{m+2}+\theta_4
    \end{align}
The maximum centroid error over the $m$-th sequence is derived in a similar fashion, the details of which are found in Appendix~\ref{app:main-res}, leading to the recursion:
\begin{align}\label{eqn:max-cent-err}
    \omega_{m+2}   \leq \alpha_1^\tau\omega_{m+1}+\alpha_2\tau\chi_{m+1}+\alpha_2\tau\chi_{m+2} + \tau\alpha_3
\end{align}

We can establish a coupled recursion from~\eqref{eqn:max-cons-err} and~\eqref{eqn:max-cent-err} (see Appendix~\ref{app:main-res}):\\~\\
\begin{align}\label{eqn:coupled-recursion}
    \begin{bmatrix}
        \omega_m \\
        \chi_m
    \end{bmatrix} &\leq \begin{bmatrix}
        \alpha_1^\tau & \alpha_2\tau(1+\gamma\theta_1) \\
        \frac{3}{2}\tau\theta_3 & \frac{3}{2}(\theta_1+\tau\theta_2)
    \end{bmatrix} \begin{bmatrix}
        \omega_{m-1} \\
        \chi_{m-1}
    \end{bmatrix}\mathlinebreak\qquad+\begin{bmatrix}
        \tau\alpha_3 \\
        \frac{3}{2}\theta_4 
    \end{bmatrix}+\mathcal{O}(\mu^3)
\end{align} 

From~\eqref{eqn:coupled-recursion} we derive the mean-square deviation (MSD) for the recursion, defined as $\expectnorm{\w_i-\one\otimes w^o}$.\\~\\

\begin{theorem}\label{thm:main-res}
\begin{subequations}
    Under Assumptions~\eqref{assump:regularity}-\eqref{assump:comb-mat} and for ${\mu\leq \min\left( \frac{\nu\sqrt{K}(1-\epstau)}{72\tau\delta\sqrt{\delta^2+\beta^2}}, \frac{2}{\nu\tau}\right)}$ the maximum expected MSD over the $m$-th FTC sequence for the Aug-DGM recursion is given by:
    \begin{align}
        &\max_{i \in \mathcal{S}_m} \expectnorm{\w_i-\one\otimes w^o} \leq \gamma^{m-1}c_1+\mathlinebreak\qquad\mathcal{O}\Bigg( \frac{\mu(\sigma^2+\beta^2\zeta^2)}{\nu K} +\frac{\mu^2\tau^2\delta^2\left(1+\epstau\right)(\sigma^2+\beta^2\zeta^2)}{\nu^2 K\left(1-\epstau\right)^2}\mathlinebreak\qquad +\frac{\mu^2\tau^2(1+\epstau)(\sigma^2+\beta^2\zeta^2)}{(1-\epstau)^2}\Bigg)+\mathcal{O}\left(\mu^3\right) \label{eq:thm:main-res}
    \end{align}
    where \textcolor{black}{${\gamma=1-\frac{\tau\nu\mu}{8}+\frac{3\tau\nu\epstau(1+\epstau)\mu}{32}+\mathcal{O}(\mu^3)}$ for ${0\leq\epstau\leq\frac{3}{4}}$}. Here $c_1$ represents the initial consensus error and centroid error terms in the first FTC period: ${c_1\triangleq2\kappa\sqrt{2}(\omega_1+\chi_1)}$ for $\kappa\in\mathbb{R}_{\geq 1}$. Details of the proof are found in Appendix \ref{app:main-res}.
    \end{subequations}
\end{theorem}
The bound in~\eqref{eq:thm:main-res} bears resemblance to other decentralized algorithm bounds with the notable exception that the rate of convergence $\gamma$ is faster by a factor of $\tau$. The increase by a factor of $\tau$ arises from the fact that we study the recursion in steps of $\tau$ when considering the maximum error over the sequence. \textcolor{black}{If we amortize the convergence rate per individual iteration, this yields a rate of ${1-\Theta(\nu\mu)+\Theta(\nu\epstau\mu)}$}. Further discussion on the bound is provided in Sec.~\ref{sec:discussion}.

\vspace{-5pt}
\section{Simulations and Discussion}\label{sec:discussion}

\begin{center}
\begin{figure}[]
    \centering % <-- added
  \includegraphics[width=0.95\linewidth]{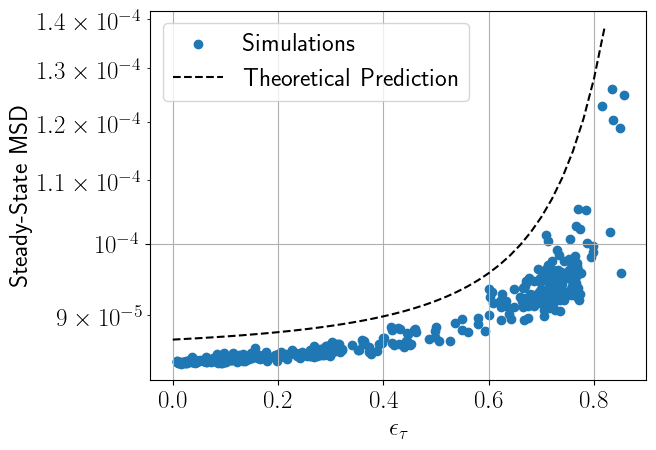}
\caption{\textcolor{black}{Steady-state error against $\epstau$ for randomly generated graphs (with $\tau$ fixed to 5) demonstrating the deterioration in performance with increasing values of $\epstau$. To simplify comparison of the trend the theoretical prediction has been shifted downwards.}}
\label{fig:msd-epstau}
\end{figure}
\end{center}
% \begin{center}
% \begin{figure}[]
%     \centering % <-- added
%   \includegraphics[width=0.9\linewidth]{Images/MSD.png}
% \caption{MSD for a path graph with \textcolor{black}{16}  agents demonstrating the increasing steady-state error with increasing $\epstau$. The deterioration in performance is graceful, indicating that Aug-DGM is robust to inaccuracies in the FTC sequence.}
% \label{fig:stochastic}
% \end{figure}
% \end{center}

\vspace{-10pt}
In this section, we numerically verify some of the key trends predicted by the performance guarantees in Theorem~\ref{thm:main-res}. We consider a linear regression problem, where each agent is provided with $N$ labeled samples generated via a linear model ${\boldsymbol{\gamma}_{k,n}=\boldsymbol{h}_{k,n}\transp w^o+\boldsymbol{v}_n}$ with a least-squares objective:
\begin{align}
    J_k(w) = \frac{1}{2N}\sum_{n=1}^N (\gamma_{k,n}-h_{k,n}\transp w)^2
\end{align}
Here, the features $\boldsymbol{h}_{k,n}$ follow a normal distribution ${\boldsymbol{h}_{k,n}\sim\mathcal{N}(0,I_M)}$ while the noise terms are sampled as ${\boldsymbol{v}_{k,n}\sim\mathcal{N}(0,0.1)}$. We set $M=20$ and $N=30$.
Optimization is carried out using the Aug-DGM algorithm in~\eqref{eq:aug-dgm}. Stochastic gradients are employed by selecting a random sample $\boldsymbol{n}_i$ at each iteration $i$ and evaluating its gradient:
\begin{align}
    \widehat{\nabla J}_k(\wb_{k,i})=h_{k,\boldsymbol{n}_i}( h_{k,\boldsymbol{n}_i}\transp\wb_{k,i}-\gamma_{k,\boldsymbol{n}_i})
\end{align}

\textcolor{black}{Fig.~\ref{fig:msd-epstau} demonstrates the steady-state MSD against $\epstau$ for the optimization problem on randomly generated 8-node graphs. The graphs are sampled from three families: Erdős–Rényi (binomial) graphs with an edge probability ${p\sim U(0.1,0.9)}$, Barabasi-Albert graphs~\cite{albert02} with edge numbers ${m\sim U(1,6)}$, or random 8-node trees. Graph connectivity is verified, otherwise a new graph is sampled. The approximate FTC matrices are then generated using the iterative method described in~\cite{eusipco} with $\tau$ is set to 5. The resulting values of $\epstau$, vary depending on the solver's progress and whether the graph's actual consensus number is 5. This setup provides a realistic example of the numerical errors typically associated with generating FTC matrices. These learned approximate FTC matrices are then used in the Aug-DGM algorithm with stochastic gradients and $\mu=0.01$. In total, 1000 independent simulations were performed. The results in Fig.~\ref{fig:msd-epstau} clearly demonstrate the increased steady-state error with higher values of $\epstau$, corroborating the theoretical prediction that the MSD is proportional to ${\mathcal{O}\left(\frac{1}{(1-\epstau)^2}\right)}$. We note that the restriction of $\epstau$ to the range ${[0,0.75]}$ is needed to establish the analytical guarantees, but we observe in simulations that Aug-DGM can converge even for larger \( \epsilon_{\tau} \).}

% Fig.~\ref{fig:stochastic} presents the MSD for the optimization problem on a path graph with 16 agents. \textcolor{black}{As before, the MSD is defined as ${\expectnorm{\w_i-\one\otimes w^o}}$}. Three sets of combination matrices are used: one with perfect FTC ($\epstau=0$) and the remaining two with approximate FTC (${\epstau=0.3}$ and ${\epstau=0.6}$ respectively). The weighting rule is provided in the supplementary material. The approximate FTC sequences are generated by adding noise to the non-zero elements in the FTC matrices and adjusting the diagonal entries so that the matrices remain doubly-stochastic. The step size for all three is set to ${8\cdot 10^{-3}}$. The results clearly demonstrate the increased steady-state error with higher values of $\epstau$, corroborating the theoretical prediction that the MSD grows with ${\mathcal{O}\left(\frac{\epstau}{(1-\epstau)^2}\right)}$. We note that the restriction of $\epstau$ to the range ${[0,0.75]}$ is needed to establish the analytical guarantees, but we have observed in simulations that Aug-DGM can converge even for larger \( \epsilon_{\tau} \).

\begin{center}
\begin{figure}[]
    \centering % <-- added
  \includegraphics[width=0.9\linewidth]{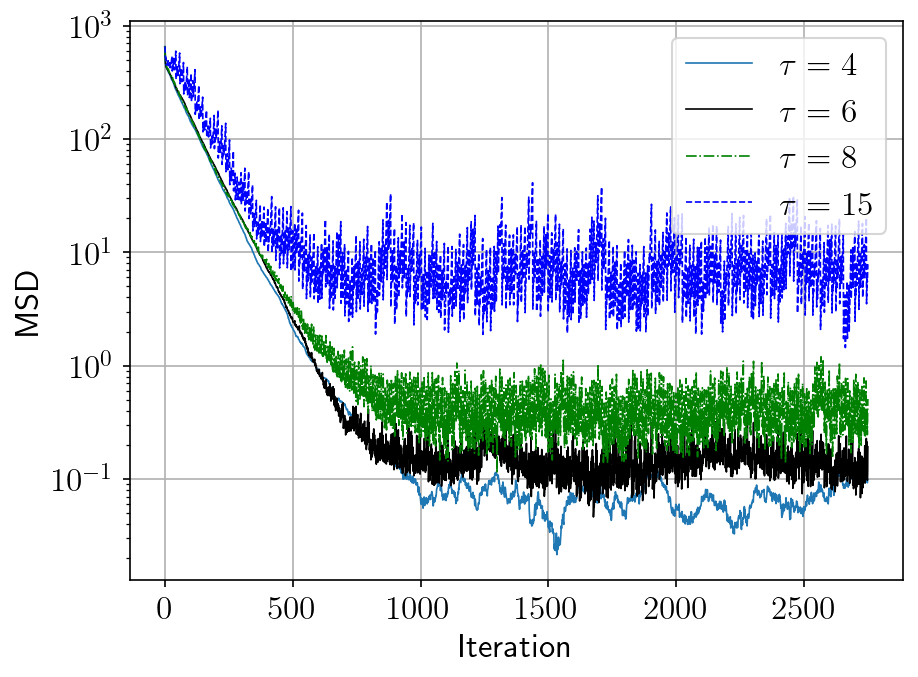}
\caption{\textcolor{black}{MSD for graphs with the same number of agents (${K=16}$), but varying values of $\tau$ demonstrating the performance deterioration from higher values of $\tau$.}}
\label{fig:incr-tau}
\end{figure}
\end{center}
\vspace{-20pt}
Increased values of $\tau$ also degrade the performance. The MSD in~\eqref{eq:thm:main-res} has a steady-state error that depends on \textcolor{black}{${\mathcal{O}(\mu^2\tau^2)}$} which we demonstrate in Fig.~\ref{fig:incr-tau}. \textcolor{black}{Here, we examine the same least-squares optimization problem on different 16-node graphs, where the exact FTC sequences are known ($\epstau=0)$: namely, a 16-hypercube~\cite{nguyen-gt}, $4\times4$ square-grid graph, 16-cycle, and 16-path. The FTC sequences for the cycle and path graphs are provided in the supplementary material. The grid graph's sequence is the concatenation of the FTC sequence for the 4-node paths acting first along the horizontal and then the vertical dimensions (leveraging the fact that the $4\times4$ square-grid graph is the Cartesian product of two 4-paths). The results in~\eqref{eq:thm:main-res} confirm that increasing $\tau$ worsens the performance, consistent with the theoretical predictions. }

We attribute the performance degradation caused by larger values of $\tau$ to the drift between agents’ local models. The $\tau-$connectivity (Assumption~\ref{assump:comb-mat}) of FTC sequences means they only average effectively over the full sequence. Drift between agents' local models may occur during the sequence, deteriorating the performance and requiring a decreased step size to compensate. This is consistent with the condition that ${\mu\leq \mathcal{O}\left(\frac{1}{\tau}\right)}$. \textcolor{black}{Limiting the drift is the reason for the variant of Aug-DGM adopted, in which the auxiliary update of~\eqref{eqn:aug-dgm-node} applies $\mu$ to the gradients, $\nabla J_k(w)$. The gradient terms perturb the recursion, driving agents towards their locally optimal model and encouraging neighboring models to differ. Scaling the gradients by the step-size $\mu$ helps reduce the accumulation of errors by reducing this local perturbation.} 

Finally, we observe an interesting trade-off between $\tau$ and $\epstau$. Exact FTC sequences prove especially effective when the consensus number is relatively low, such as in the hypercube in Fig.~\ref{fig:tradeoff}(a) where $K=8, \tau=3$. There is a demonstrable benefit to using the exact FTC sequence over a classical weighting rule. For graphs with a larger consensus number, the performance degrades. In this case we may choose to set $\tau$ to be smaller than the actual consensus number. The approximate FTC sequence will have a larger approximation error $\epstau$, but a smaller value of $\tau$. Considering that the MSD depends on the ratio $\frac{\tau^2}{(1-\epstau)^2}$ this may be beneficial. We demonstrate this in the path graph in Fig.~\ref{fig:tradeoff}(b). The perfect FTC sequence on the path suffers from increasing drift over the course of the sequence. We therefore see the error grow over the sequence, sharply decreasing every $\tau=7$ iterations and demonstrating the two timescale behavior previously described. The best performance is achieved by the approximate FTC case \textcolor{black}{($\epstau=0.4$). Matrices are learned using the method in~\cite{eusipco} with $\tau=3$. The} approximate case outperforms both the perfect FTC case \textcolor{black}{and the static metropolis rule}. \textcolor{black}{ More generally we can conclude from our theoretical predictions and empirical results that approximate FTC sequences provide a benefit over a static weighting rule (with second largest eigenvalue $\lambda$) when $\frac{\tau^2}{(1-\epstau)^2}$ is small relative to $\frac{1}{(1-\lambda)^2}$}. %, where $\tau=1$ and $\epstau$ is equal to the second largest eigenvalue, $\lambda=0.95$}. 

\begin{figure}
\centering
\includegraphics[width=0.4\textwidth]{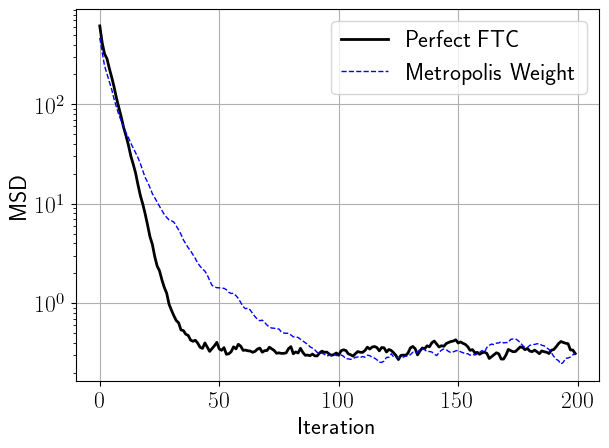}
\caption*{\centering \small(a) Hypercube}
\includegraphics[width=0.4\textwidth]{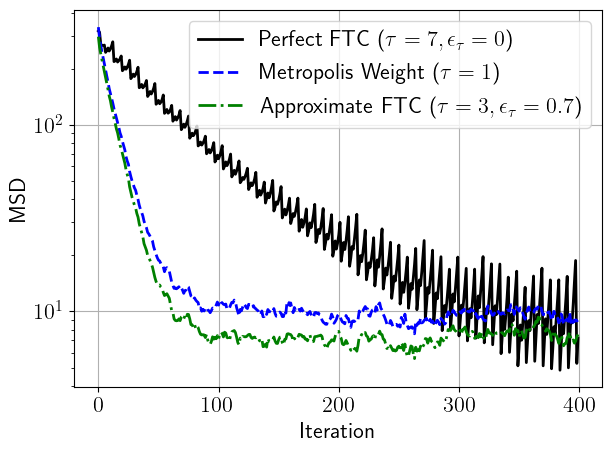}
\caption*{\centering \small(b) Path}
\caption{\textcolor{black}{Performance of (approximate) FTC sequences on different graphs.} }\label{fig:tradeoff}
\end{figure}

\section{Conclusion}
In this work we analyzed the convergence properties of Aug-DGM with approximate FTC sequences. Our results confirm that approximate FTC sequences can provide performance benefits to gradient-tracking algorithms, showing that the performance degrades gracefully as the approximation error increases. The theoretical bounds and subsequent simulations also highlight the non-trivial relationship between the approximation accuracy \( \epsilon_{\tau} \) and length of the sequence \( \tau \). While the accuracy generally increases as the length of the sequence increases, this comes at a cost of parameter drift, which can deteriorate performance.

\bibliographystyle{IEEEtran}
\bibliography{References}

% \onecolumn
{\appendices

\section{Transformation}\label{app:transformation}

\subsection{First Transformation}
    Beginning with~\eqref{eq:aug-dgm:w} and noting that ${\g_{i}=\y_i+\mu\A_i\gradappcal(\w_i)}$, the update for $\w_i$ becomes:
    \begin{align}
        \w_{i}&= \A_i\left(\w_{i-1}-\y_{i-1}-\mu\A_{i-1}\gradappcal(\w_{i-1})\right) \nonumber \\
        &= \A_{i}\w_{i-1}-\A_i\y_{i-1}-\mu\A_i\A_{i-1}\gradappcal(\w_{i-1})\label{eq:aug-dgm-y:w}
    \end{align}
    Applying the change of variable to both sides of~\eqref{eq:aug-dgm:g}:
    \begin{align}
        \y_i+\mu\A_i\gradappcal(\w_i)&=\A_{i}(\y_{i-1}+\mu\A_{i-1}\gradappcal(\w_{i-1}))\nonumber\\&\quad+\mu\A_i\gradappcal(\w_{i})-\mu\A_i\gradappcal(\w_{i-1}) \nonumber
        % \y_i &= \A_i\y_{i-1}+\mu\A_i\A_{i-1}\gradappcal(\w_{i-1})-\mu\A_i\gradappcal(\w_{i-1}) \nonumber \\
        \end{align}
        and after rearranging:
        \begin{align}
        \y_{i}&=\A_{i}\y_{i-1}-\mu\A_i(I-\A_{i-1})\gradappcal(\w_{i-1}) \label{eq:aug-dgm-y:y}
    \end{align}
    % The initialization follows by defining $A_0=I$ so that,
    % \begin{align}
    %     \y_0=\g_0-\mu\A_0\gradappcal(\w_0) = \mu\gradappcal(\w_0)-\mu\gradappcal(\w_0) = 0 \nonumber
    % \end{align}

\subsection{Second Transformation}
From~\eqref{eq:aug-dgm-y:w} and using ${\z_i\triangleq\y_i+\mu \A_i\gradcal(\w_{c,i})}$:
    \begin{align}
        % \w_{i}&=\A_{i}\w_{i-1}-\A_i\y_{i-1}-\mu\A_i\A_{i-1}\gradappcal(\w_{i-1})\nonumber  \\
        \w_{i}&= \A_{i}\w_{i-1}-\A_i(\z_{i-1}-\mu\A_{i-1}\gradcal(\w_{c,i-1}))\nonumber\\&\quad-\mu\A_i\A_{i-1}\gradappcal(\w_{i-1})\nonumber  \\
        % &\stackrel{(a)}{=}  \A_{i}\w_{i-1}-\A_i\z_{i-1}+\mu\A_i\A_{i-1}\gradcal(\w_{c,i-1})-\mu\A_i\A_{i-1}(\gradcal(\w_{i-1})+\s_i(\w_{i-1})) \nonumber \\
        &= \A_{i}\w_{i-1}-\A_i\z_{i-1}+\mu\A_i\A_{i-1}(\gradcal(\w_{c,i-1})\nonumber\\&\quad-\gradcal(\w_{i-1}))-\mu\A_i\A_{i-1}\s_{i}(\w_{i-1}) \nonumber 
    \end{align}
    where in the last line we used the definition of gradient noise $\s_i(\w_{i-1})\triangleq \gradappcal(\w_{i-1}) - \gradcal(\w_{i-1})$.

    From the definition of $\z_i$, we have ${\y_i=\z_i-\mu\A_i\gradcal(\w_{c,i})}$. Replacing $\y_i$ and $\y_{i-1}$ in~\eqref{eq:aug-dgm-y:y}:
    \begin{align}
        \y_{i}&=\A_{i}\y_{i-1}-\mu\A_i(I-\A_{i-1})\gradappcal(\w_{i-1}) \nonumber \\
        \z_{i}&-\mu\A_{i}\gradcal(\w_{c,i}) = \A_i( \z_{i-1}-\mu\A_{i-1}\gradcal(\w_{c,i-1})) \nonumber \\ &-\mu\A_i(I-\A_{i-1})\gradappcal(\w_{i-1}) \nonumber \\
        \z_i &= \A_i\z_{i-1}+\mu\A_{i}\gradcal(\w_{c,i}) -\mu\A_i\A_{i-1}\gradcal(\w_{c,i-1})) \nonumber\\&\quad-\mu\A_i(I-\A_{i-1})\gradappcal(\w_{i-1})\nonumber  \\
        % &\stackrel{(b)}{=} \A_i\z_{i-1}+\mu\A_{i}\gradcal(\w_{c,i}) -\mu\A_i\A_{i-1}\gradcal(\w_{c,i-1})) -\mu\A_i(I-\A_{i-1})\gradappcal(\w_{i-1}) + \mu\A_i\gradcal(\w_{c,i-1}) \lbreak -\mu\A_i\gradcal(\w_{c,i-1}) \nonumber \\
        &\stackrel{(a)}{=} \A_i\z_{i-1} +\mu\A_i(I-\A_{i-1})\gradcal(\w_{c,i-1})) \nonumber\\&\quad-\mu\A_i(I-\A_{i-1})\gradappcal(\w_{i-1})\nonumber\\&\quad+\mu\A_i(\gradcal(\w_{c,i})-\gradcal(\w_{c,i-1})) \nonumber \\
        % &= \A_i\z_{i-1} -\mu\A_i(I-\A_{i-1})(\gradappcal(\w_{i-1})-\gradcal(\w_{c,i-1}))) +\mu\A_i(\gradcal(\w_{c,i})-\gradcal(\w_{c,i-1})) \nonumber \\
        &\stackrel{(b)}{=} \A_i\z_{i-1} -\mu\A_i(I-\A_{i-1})(\gradcal(\w_{i-1})+\s_i(\w_{i-1})-\nonumber\\&\quad\gradcal(\w_{c,i-1}))) +\mu\A_i(\gradcal(\w_{c,i})-\gradcal(\w_{c,i-1})) \nonumber \\
        &= \A_i\z_{i-1} -\mu\A_i(I-\A_{i-1})(\gradcal(\w_{i-1})-\gradcal(\w_{c,i-1}))) \nonumber\\&\quad+\mu\A_i(\gradcal(\w_{c,i})-\gradcal(\w_{c,i-1}))\nonumber\\&\quad-\mu\A_i(I-\A_{i-1})\s_i(\w_{i-1})
    \end{align}
    where in $(a)$ we added and subtracted $\mu\A_i\gradcal(\w_{c,i-1})$ and grouped like terms and in $(b)$ we used the definition of gradient noise $\s_i(\w_{i-1})\triangleq \gradappcal(\w_{i-1}) - \gradcal(\w_{i-1})$.

\section{Consensus Error Relation}\label{app:cons-err-relation}

\subsection{Proof of Lemma~\ref{lemma:v-h-bound}}
Considering first, $\vb_i$, and applying the vector norm to the block-structure definition in~\eqref{eqn:v-defn}:
    \begin{align}
    &\expectnormfi{\vb_i} \nonumber\\
    &= \expectnormfi{\Ahat_{i}\Ahat_{i-1}\s_{i}(\w_{i-1})}\nonumber\\&\quad+\expectnormfi{\Ahat_{i}(I-\A_{i-1})\s_{i}(\w_{i-1})} \nonumber\\
    &\stackrel{(a)}{\leq} \norm{\Ahat_{i}}^2\norm{\Ahat_{i-1}}^2\expectnormfi{\s_{i}(\w_{i-1})}\nonumber\\&\quad+\norm{\Ahat_{i}}^2\norm{I-\Ahat_{i-1}}^2\expectnormfi{\s_{i}(\w_{i-1})} \nonumber\\
    &\leq \expectnormfi{\s_{i}(\w_{i-1})}+2\expectnormfi{\s_{i}(\w_{i-1})} \nonumber\\
    &\stackrel{\eqref{eq:stacked-gradnoise-var}}{\leq} 3\Big(3\beta^2\zeta^2+3\beta^2K\norm{\centerr_{c,i-1}}^2+3\beta^2\norm{\what_{i-1}}^2+\sigma^2 \Big)
\end{align}
where in $(a)$ we used the submultiplicativity of the matrix norm \textcolor{black}{and the fact that the spectral norm of $\Ahat_i$ is 1, since ${\Ahat_i\triangleq A_i\otimes I_M}$ has the same spectral norm as $A_i$ (Assumption~\ref{assump:comb-mat})}.

Next, considering $\hb_i$ from~\eqref{eqn:h-defn}:
\begin{align}
    &\norm{\hb_i}^2 \nonumber\\&= \norm{\Ahat_i\Ahat_{i-1}(\gradcal(\w_{i-1})-\gradcal(\w_{c,i-1}))}^2\mathlinebreak\quad+\Big\lVert\A_{i} (I-A_{i-1})(\gradcal(\w_{i-1})-\gradcal(\w_{c,i-1}))\nonumber\\&\quad -\Ahat_{i}(\gradcal(\w_{c,i})-\gradcal(\w_{c,i-1}))\Big\rVert^2 \nonumber\\
    &\stackrel{(a)}{\leq} \norm{\Ahat_i}^2\norm{\Ahat_{i-1}}^2\norm{\gradcal(\w_{i-1})-\gradcal(\w_{c,i-1})}^2\nonumber\\&\quad+\big\lVert \A_{i} (I-A_{i-1})(\gradcal(\w_{i-1})-\gradcal(\w_{c,i-1}))\mathlinebreak\quad-\Ahat_{i}(\gradcal(\w_{c,i})-\gradcal(\w_{c,i-1}))\big\rVert^2 \nonumber\\
    &\leq\Big\lVert\gradcal(\w_{i-1})-\gradcal(\w_{c,i-1})\Big\rVert^2\mathlinebreak\quad+\Big\lVert\A_{i} (I-A_{i-1})(\gradcal(\w_{i-1})-\gradcal(\w_{c,i-1}))\mathlinebreak\quad-\Ahat_{i}(\gradcal(\w_{c,i})-\gradcal(\w_{c,i-1}))\Big\rVert^2 \nonumber\\
    &\stackrel{\eqref{eq:assump-lipsch}}{\leq}\delta^2\norm{\w_{i-1}-\w_{c,i-1}}^2 + 2\Big\lVert\A_i(I-\A_{i-1})(\gradcal(\w_{i-1})\mathlinebreak\quad-\gradcal(\w_{c,i-1}))\Big\rVert^2+2\norm{\Ahat_i(\gradcal(\w_{c,i})-\gradcal(\w_{c,i-1}))}^2 \nonumber\\
    &\stackrel{(b)}{\leq}\delta^2\norm{\w_{i-1}-\w_{c,i-1}}^2\mathlinebreak\quad+ 2\norm{\A_i}^2\Big\lVert I-\A_{i-1}\big\rVert^2\norm{\gradcal(\w_{i-1})-\gradcal(\w_{c,i-1})}^2 \mathlinebreak\quad +2\norm{\Ahat_i}^2\norm{\gradcal(\w_{c,i})-\gradcal(\w_{c,i-1})}^2 \nonumber\\
    &\leq\delta^2\norm{\w_{i-1}-\w_{c,i-1}}^2 \mathlinebreak\quad+ 4\norm{\gradcal(\w_{i-1})-\gradcal(\w_{c,i-1})}^2\mathlinebreak\quad+2\norm{\gradcal(\w_{c,i})-\gradcal(\w_{c,i-1})}^2 \nonumber\\
    &\stackrel{\eqref{eq:assump-lipsch}}{\leq} \delta^2\norm{\w_{i-1}-\w_{c,i-1}}^2 + 4\delta^2\norm{\w_{i-1}-\w_{c,i-1}}^2\mathlinebreak\quad+2\delta^2\norm{\w_{c,i}-\w_{c,i-1}}^2 \nonumber\\
    &= 5\delta^2\norm{\what_{i-1}}^2+2\delta^2K\norm{\wb_{c,i}-\wb_{c,i-1}}^2
    \end{align}
    where $(a)$ and $(b)$ used the submultiplicative property of the matrix norm. From the centroid recursion in~\eqref{eqn:centroid-evolution}:
    \begin{align}
        \wb_{c,i}-\wb_{c,i-1}=-\mu\onekronim(\gradcal(\w_{i-1})+\s_{i}(\w_{i-1}))
    \end{align}
    and thus,
    \begin{align}
        &\expectnormfi{\hb_{i}} \nonumber\\&\leq 5\delta^2\expectnormfi{\what_{i-1}}\mathlinebreak\quad+2\delta^2K\mathbb{E}\Bigg[\Bigg\lVert\mu\onekronim(\gradcal(\w_{i-1})\mathlinebreak\qquad\qquad\qquad+\s_{i}(\w_{i-1}))\Bigg\lVert^2\Big\vert \w_{i-1}\Bigg] \nonumber\\
        &= 5\delta^2\norm{\what_{i-1}} \mathlinebreak\quad+ 2\delta^2K\mu^2\expectnormfi{\onekronim(\gradcal(\w_{i-1})}\mathlinebreak\quad+ 2\delta^2K\mu^2\expectnormfi{ \onekronim\s_{i}(\w_{i-1}))}\mathlinebreak\quad+ 2\delta^2K\mu^2\expect{\s_{i}(\w_{i-1})\transp\onekronim\gradcal(\w_{i-1}) \big|\w_{i-1}} \nonumber\\
        &\stackrel{\eqref{eq:assump-gradnoise-unbias}}{=} 5\delta^2\norm{\what_{i-1}} + 2\delta^2K\mu^2\norm{\onekronim(\gradcal(\w_{i-1})}^2\mathlinebreak\quad+ 2\delta^2K\mu^2\expectnormfi{ \onekronim\s_{i}(\w_{i-1}))} \nonumber\\
        &\stackrel{\eqref{eq:avg-stacked-gradnoise-var}}{\leq} 5\delta^2\norm{\what_{i-1}} + 2\delta^2K\mu^2\norm{\onekronim(\gradcal(\w_{i-1})}^2\mathlinebreak\quad + 2\delta^2\mu^2(3\beta^2\zeta^2+3\beta^2K\norm{\centerr_{c,i-1}}^2+3\beta^2\norm{\what_{i-1}}^2+\sigma^2) \nonumber\\
        &\stackrel{(a)}{=} 5\delta^2\norm{\what_{i-1}} \mathlinebreak\quad+ 2\delta^2K\mu^2\Big\lVert\onekronim((\gradcal(\w_{i-1})-\gradcal(\w_{c,i-1}))\mathlinebreak\qquad\qquad\qquad+ \onekronim\gradcal(\w_{c,i-1})\Big\rVert^2 \mathlinebreak\quad+ 2\delta^2\mu^2(3\beta^2\zeta^2+3\beta^2K\norm{\centerr_{c,i-1}}^2+3\beta^2\norm{\what_{i-1}}^2+\sigma^2) \nonumber\\
        &\stackrel{(b)}{\leq} 5\delta^2\norm{\what_{i-1}} \mathlinebreak\quad+ 4\delta^2K\mu^2\norm{\onekronim((\gradcal(\w_{i-1})-\gradcal(\w_{c,i-1}))}^2 \mathlinebreak\quad+4\delta^2K\mu^2\norm{\onekronim\gradcal(\w_{c,i-1})} ^2\mathlinebreak\quad+2\delta^2\mu^2(3\beta^2\zeta^2+3\beta^2K\norm{\centerr_{c,i-1}}^2+3\beta^2\norm{\what_{i-1}}^2+\sigma^2) \nonumber\\
        &\stackrel{(c)}{\leq} 5\delta^2\norm{\what_{i-1}} + 4\delta^4\mu^2\norm{\what_{i-1}}^2\mathlinebreak\quad+4\delta^2K\mu^2\norm{\onekronim\gradcal(\w_{c,i-1})}^2 \mathlinebreak\quad +2\delta^2\mu^2(3\beta^2\zeta^2+3\beta^2K\norm{\centerr_{c,i-1}}^2+3\beta^2\norm{\what_{i-1}}^2+\sigma^2) \nonumber\\
        &\stackrel{(d)}{\leq} 5\delta^2\norm{\what_{i-1}} + 4\delta^4\mu^2\norm{\what_{i-1}}^2\mathlinebreak\quad+8\delta^4\mu^2\norm{\centerr_{c,i-1}}^2+8\delta^4\mu^2\zeta^2 \mathlinebreak\quad +2\delta^2\mu^2(3\beta^2\zeta^2+3\beta^2K\norm{\centerr_{c,i-1}}^2+3\beta^2\norm{\what_{i-1}}^2+\sigma^2) \nonumber\\
        &\leq (5\delta^2+4\delta^4\mu^2+6\beta^2\delta^2\mu^2)\norm{\what_{i-1}} \mathlinebreak\quad+ 2\delta^2\mu^2(4\delta^2+4\beta^2K)\norm{\centerr_{c,i-1}}^2+6\beta^2\delta^2\mu^2\zeta^2+2\delta^2\mu^2\sigma^2
    \end{align}
    where in $(a)$ we added and subtracted ${\onekronim\gradcal(\w_{c,i-1})}$, $(b)$ used Jensen's inequality to split the terms in the expectation, $(c)$ follows from:
    \begin{align}
    &\norm{\onekronim(\gradcal(\w_i)-\gradcal(\w_{c,i}))}^2 \nonumber\\&= \norm{\frac{1}{K}\sum_{k=1}^K (J_k(\wb_{k,i})-J_k(\wb_{c,i})) }^2 \nonumber\\
        &\leq \frac{1}{K}\sum_{k=1}^K\norm{J_k(\wb_{k,i})-J_k(\wb_{c,i}) }^2 \nonumber\\
        &\leq\frac{\delta^2}{K}\sum_{k=1}^K\norm{\wb_{k,i}-\wb_{c,i} }^2 \nonumber\\
        &= \frac{\delta^2}{K}\norm{\what_i}^2 
    \end{align}
    
    and $(d)$ from,
    \begin{align}
    &\norm{\onekronim\gradcal(\w_{c,i-1})}^2  \nonumber\\
        &= \norm{\frac{1}{K}\sum_{k=1}^KJ_k(\wb_{c,i-1})}^2 \nonumber\\
        &\leq\frac{1}{K} \sum_{k=1}^K\norm{J_k(\wb_{c,i-1})}^2 \nonumber\\
        & \leq\frac{1}{K}\sum_{k=1}^K\norm{J_k(\wb_{c,i-1})-J_k(w^o)+J_k(w^o)}^2 \nonumber\\
        &\leq \frac{2}{K}\sum_{k=1}^K (\norm{J_k(\wb_{c,i-1})-J_k(w^o)}^2+\norm{J_k(w^o)}^2) \nonumber\\
        &= \frac{2}{K}\sum_{k=1}^K (\norm{J_k(\wb_{c,i-1})-J_k(w^o)}^2+\norm{J_k(w^o)-J_k(w_k^o)}^2) \nonumber\\
        &\leq \frac{2}{K}\sum_{k=1}^K (\norm{J_k(\wb_{c,i-1})-J_k(w^o)}^2+\delta_k^2\norm{w^o-w_k^o}^2) \nonumber\\
        &\leq \frac{2\delta^2}{K}\norm{\centerr_{c,i-1}}+\frac{2\delta^2}{K}\sum_{k=1}^K\zeta_k^2 \nonumber\\
        &\leq \frac{2\delta^2}{K}\norm{\centerr_{c,i-1}}+\frac{2\delta^2}{K}\zeta^2
    \end{align}
    The final bound is obtained by setting $\mu^2\leq\frac{1}{4\delta^2}$:
    \begin{align}
        &\expectnormfi{\hb_i}\leq 3\left(2\delta^2+\frac{1}{2}\beta^2\right)\norm{\what_{i-1}} \mathlinebreak\quad\quad+ 2(\delta^2+\beta^2K)\norm{\centerr_{c,i-1}}^2+\frac{3}{2}\beta^2\zeta^2+\frac{1}{2}\sigma^2
    \end{align}
    
\subsection{Proof of Lemma~\eqref{lemma:consensus-bound}}

The consensus error is bounded by separating out $\vb_i$ on the assumption that the gradient noise is unbiased and pairwise uncorrelated (Assumption~\eqref{assump:gradnoise}) and then using a vector norm bound to separate $\hb_i$ as in~\cite{koloskova}:
\begin{align}
    &\expectnormfi{\xhat_i} \nonumber\\&= \mathbb{E}\Bigg[\Big\lVert G_{i:m\tau+1}\xhat_{m\tau}-\mu\sum_{\mathclap{j=m\tau+1}}^{i-1}G_{i:j+1}\hb_j\mathlinebreak\quad-\mu\sum_{\mathclap{j=m\tau+1}}^{i-1}G_{i:j+1}\vb_j-\mu \hb_i-\mu \vb_i \Big\rVert^2 \Big\vert \w_{i-1}\Bigg] \nonumber\\
    &\stackrel{(a)}{=} \mathbb{E}\Bigg[\Big\lVert G_{i:m\tau+1}\xhat_{m\tau}-\mu\sum_{\mathclap{j=m\tau+1}}^{i-1}G_{i:j+1}\hb_j\mathlinebreak\quad-\mu\sum_{\mathclap{j=m\tau+1}}^{i-1}G_{i:j+1}\vb_j-\mu \hb_i\Big\rVert^2\Big\vert \w_{i-1}\Bigg] +\expectnormfi{\mu \vb_i }\nonumber \\
    &\stackrel{(b)}{\leq} (1+\alpha_1)\mathbb{E}\Bigg[\Big\lVert G_{i:m\tau+1}\xhat_{m\tau}-\mu\sum_{\mathclap{j=m\tau+1}}^{i-1}G_{i:j+1}\hb_j\mathlinebreak\quad-\mu\sum_{\mathclap{j=m\tau+1}}^{i-1}G_{i:j+1}\vb_j\Big\rVert^2\Big\Vert \w_{i-1}\Bigg]\mathlinebreak\quad +(1+\alpha_1^{-1})\expectnormfi{\mu \hb_i}+\expectnormfi{\mu \vb_i } 
\end{align}
where in $(a)$ the cross-terms were removed based on Assumption~\eqref{assump:gradnoise}, and $(b)$ follows from Young's inequality for vectors. Setting $\alpha_1=\frac{1}{\gamma-1}$ for some constant $\gamma>1$ that will be chosen later.
\begin{align}
    &\expectnormfi{\xhat_i} \nonumber\\&\leq \frac{\gamma}{\gamma-1}\mathbb{E}\Bigg[\Big\lVert G_{i:m\tau+1}\xhat_{m\tau}\mathlinebreak\qquad\qquad-\mu\sum_{\mathclap{j=m\tau+1}}^{i-1}G_{i:j+1}\hb_j-\mu\sum_{\mathclap{j=m\tau+1}}^{i-1}G_{i:j+1}\vb_j\Big\rVert^2\Big\vert\w_{i-1}\Bigg] \mathlinebreak\quad+\gamma\expectnormfi{\mu \hb_i}+\expectnormfi{\mu \vb_i } 
\end{align}
We can repeat the process, separating out $\vb_{i-1}$ and $\hb_{i-1}$:
\begin{align}
    &\expectnormfi{\xhat_i} \nonumber\\&\leq \frac{\gamma}{\gamma-1}\mathbb{E}\Bigg[\Big\lVert G_{i:m\tau+1}\xhat_{m\tau}-\mu\sum_{\mathclap{j=m\tau+1}}^{i-2}G_{i:j+1}\hb_j\mathlinebreak\quad-\mu\sum_{\mathclap{j=m\tau+1}}^{i-2}G_{i:j+1}\vb_j-\mu G_i\hb_{i-1}-\mu G_i \vb_{i-1}\Big\rVert^2\Big\vert\w_{i-1}\Bigg] \mathlinebreak\quad+\gamma\expectnormfi{\mu \hb_i}+\expectnormfi{\mu \vb_i } \nonumber\\
    &= \frac{\gamma}{\gamma-1}\mathbb{E}\Bigg[\Big\lVert G_{i:m\tau+1}\xhat_{m\tau}-\mu\sum_{\mathclap{j=m\tau+1}}^{i-2}G_{i:j+1}\hb_j\mathlinebreak\quad-\mu\sum_{\mathclap{j=m\tau+1}}^{i-2}G_{i:j+1}\vb_j-\mu G_i\hb_{i-1}\Big\rVert^2\Big\vert\w_{i-1}\Bigg]\mathlinebreak\quad +\gamma\expectnormfi{\mu \hb_i}+\frac{\gamma}{\gamma-1}\expectnormfi{\mu G_i \vb_{i-1}}\mathlinebreak\quad+\expectnormfi{\mu \vb_i } \nonumber\\
    &\leq \frac{\gamma(1+\alpha_2)}{\gamma-1}\mathbb{E}\Bigg[\Big\lVert G_{i:m\tau+1}\xhat_{m\tau}-\mu\sum_{\mathclap{j=m\tau+1}}^{i-2}G_{i:j+1}\hb_j\mathlinebreak\quad-\mu\sum_{\mathclap{j=m\tau+1}}^{i-2}G_{i:j+1}\vb_j\Big\rVert^2\Big\vert\w_{i-1}\Bigg] +\expectnormfi{\mu \vb_i }\mathlinebreak\quad+\frac{\gamma(1+\alpha_2^{-1})}{\gamma-1}\expectnormfi{\mu G_i\hb_{i-1}}\mathlinebreak\quad+\gamma\expectnormfi{\mu  \hb_i}+\frac{\gamma}{\gamma-1}\expectnormfi{\mu G_i \vb_{i-1}} \nonumber\\
    &\leq \frac{\gamma}{\gamma-2}\mathbb{E}\Bigg[\Big\lVert G_{i:m\tau+1}\xhat_{m\tau}-\mu\sum_{\mathclap{j=m\tau+1}}^{i-2}G_{i:j+1}\hb_j\mathlinebreak\quad-\mu\sum_{\mathclap{j=m\tau+1}}^{i-2}G_{i:j+1}\vb_j\Big\rVert^2\Big\vert\w_{i-1}\Bigg] +\gamma\expectnormfi{\mu G_i \hb_{i-1}}\mathlinebreak\quad+\gamma\expectnormfi{\mu \hb_i}+\frac{\gamma}{\gamma-1}\expectnormfi{\mu G_i \vb_{i-1}}\mathlinebreak\quad+\expectnormfi{\mu \vb_i }
\end{align}

where the last line follows by setting $\alpha_2=\frac{1}{\gamma-2}$. Splitting the remainder of the terms by continuing in this fashion with \( \alpha_j = \frac{1}{\gamma - j} \):
\begin{align}
    &\expectnormfi{\xhat_i} \leq \frac{\gamma}{\gamma-(i-m\tau)}\expectnormfi{G_{i:m\tau+1}\xhat_{m\tau}}\mathlinebreak\quad+\gamma\sum_{\mathclap{j=m\tau+1}}^{i-1}\expectnormfi{\mu G_{i:j+1}\hb_j}+\gamma\expectnormfi{\mu \hb_i}\mathlinebreak\quad+\sum_{\mathclap{j=m\tau+1}~~~}^{i-1}\frac{\gamma}{\gamma-(i-j)}\expectnormfi{\mu G_{i:j+1}\vb_j}\mathlinebreak\quad+\expectnormfi{\mu \vb_i} \nonumber\\
    &\stackrel{(a)}{\leq}\frac{\gamma}{\gamma-(i-m\tau)}\norm{G_{i:m\tau+1}}^2\expectnormfi{\xhat_{m\tau}}\mathlinebreak\quad+\gamma\mu^2\sum_{\mathclap{j=m\tau+1}}^{i-1}\norm{G_{i:j+1}}^2\expectnormfi{\hb_j}\mathlinebreak\quad+\gamma\mu^2\expectnormfi{ \hb_i}+\mu^2\expectnormfi{ \vb_i}\mathlinebreak\quad+\mu^2\sum_{\mathclap{j=m\tau+1}~~~}^{i-1}\frac{\gamma\norm{G_{i:j+1}}^2}{\gamma-(i-j)}\expectnormfi{\vb_j} \nonumber\\%\label{eq:proof-cons-err-full}\\
    &\stackrel{(b)}{\leq}\frac{\gamma\epstau^2}{\gamma-(i-m\tau)}\expectnormfi{\xhat_{m\tau}}\mathlinebreak\quad+\gamma\mu^2\sum_{\mathclap{j=m\tau+1}}^{i-1}\expectnormfi{\hb_j}+\gamma\mu^2\expectnormfi{ \hb_i}\mathlinebreak\quad+\mu^2\sum_{\mathclap{j=m\tau+1}~~}^{i-1}\frac{\gamma}{\gamma-(i-j)}\expectnormfi{\vb_j}+\mu^2\expectnormfi{ \vb_i} \nonumber\\
    &= \frac{\gamma\epstau^2}{\gamma-(i-m\tau)}\norm{\xhat_{m\tau}}^2+\gamma\mu^2\sum_{\mathclap{j=m\tau+1}}^{i}\expectnormfi{\hb_j}\mathlinebreak\quad+\mu^2\sum_{\mathclap{j=m\tau+1}}^{i}\frac{\gamma}{\gamma-(i-j)}\expectnormfi{\vb_j} \label{eq:cons-err-proof-partway}
\end{align}
where in $(a)$ we used the submultiplicativity of the matrix norm, in $(b)$ we used the fact ${\norm{G_{i:m\tau+1}}^2\leq\epstau^2}$ for ${i\geq(m+1)\tau}$ which follows by considering the block upper-triangular structure of $G_{i:m\tau+1}$:
\begin{align}
    G_{i:m\tau+1} = \begin{bmatrix}
        \A_{i:m\tau+1} & (i-m\tau)\A_{i:m\tau+1}\\
        0 & \A_{i:m\tau+1}
    \end{bmatrix}
\end{align}
and then using that the spectral norm of a block upper-triangular matrix equals the maximum spectral norm of its diagonal blocks \cite{simovici}, in this case ${\norm{A_{i:m\tau+1}}=\epstau}$.\\

Setting $\gamma=2\tau\frac{1+\epstau}{1-\epstau}$ and considering the coefficient for $\expectnormfi{\vb_j}$ in~\eqref{eq:cons-err-proof-partway} we have:
\begin{align}
    \frac{\gamma}{\gamma-(i-m\tau)}&\leq\frac{\gamma}{\gamma-(2\tau-1)}
    \nonumber\\&
    =\frac{2\tau(1+\epstau)}{2\tau(1+\epstau)-(2\tau-1)(1-\epstau)}\nonumber\\&=\frac{2\tau(1+\epstau)}{1+4\tau\epstau-\epstau}\leq 2\tau 
    \end{align}
    for ${i-m\tau=i-(\left\lfloor\frac{i}{\tau}\right\rfloor-1)\tau\leq 2\tau-1<2\tau}$ and $0\leq\epstau<1$. Similarly for the coefficient of $\norm{\xhat_{m\tau}}^2$:
    \begin{align}
    \frac{\gamma\epstau^2}{\gamma-(i-m\tau)}\leq\frac{\gamma\epstau^2}{\gamma-2\tau}
    &\leq \frac{2\tau(1+\epstau)\epstau^2}{2\tau(1+\epstau)-2\tau(1-\epstau)}\nonumber\\
    &=\frac{\epstau}{2}(1+\epstau)
\end{align}

Returning to~\eqref{eq:cons-err-proof-partway}:
\begin{align}
    &\expectnormfi{\xhat_i} \leq \frac{\epstau}{2}(1+\epstau) \norm{\xhat_{m\tau}}^2\mathlinebreak\quad\quad+\frac{2\tau(1+\epstau)\mu^2}{1-\epstau}\sum_{\mathclap{j=m\tau+1}}^{i}\expectnormfi{\hb_j}\mathlinebreak\quad\quad+2\tau\mu^2\sum_{\mathclap{j=m\tau+1}}^{i}\expectnormfi{\vb_j} \nonumber\\
    &\stackrel{\eqref{eq:vi-var}}{\leq} \frac{\epstau}{2}(1+\epstau) \norm{\xhat_{m\tau}}^2+\frac{2\tau(1+\epstau)\mu^2}{1-\epstau}\sum_{\mathclap{j=m\tau+1}}^{i}\expectnormfi{\hb_j}\mathlinebreak\quad +2\tau\mu^2\sum_{\mathclap{j=m\tau+1}}^{i}\big(9\beta^2\zeta^2+9\beta^2K\norm{\centerr_{c,j-1}}^2 \mathlinebreak\qquad+9\beta^2\norm{\what_{j-1}}^2+3\sigma^2 \big) \nonumber\\
    &\stackrel{\eqref{eq:h-bound}}{\leq} \frac{\epstau}{2}(1+\epstau) \norm{\xhat_{m\tau}}^2 \mathlinebreak\quad+2\tau\mu^2\sum_{\mathclap{j=m\tau+1}}^{i}\Big(9\beta^2\zeta^2+9\beta^2K\norm{\centerr_{c,j-1}}^2 +9\beta^2\norm{\what_{j-1}}^2+3\sigma^2 \Big)\mathlinebreak\quad+\frac{2\tau(1+\epstau)\mu^2}{1-\epstau}\sum_{\mathclap{j=m\tau+1}~~}^{i}\Bigg(3\left(2\delta^2+\frac{1}{2}\beta^2\right)\norm{\what_{j-1}} \mathlinebreak\quad\quad + 2(\delta^2+\beta^2K)\norm{\centerr_{c,j-1}}^2+\frac{3}{2}\beta^2\zeta^2+\frac{1}{2}\sigma^2 \Bigg) \nonumber\\
    % &\stackrel{(a)}{=} \frac{\epstau}{2}(1+\epstau) \norm{\xhat_{m\tau}}^2 \mathlinebreak\quad+2\tau\mu^2\sum_{j=m\tau}^{i-1}\Big(9\beta^2\zeta^2+9\beta^2K\norm{\centerr_{c,j}}^2 +9\beta^2\norm{\what_{j}}^2+3\sigma^2 \Big)\mathlinebreak\quad+\frac{2\tau(1+\epstau)\mu^2}{1-\epstau}\sum_{\mathclap{j=m\tau}}^{i-1}\Bigg(3\left(2\delta^2+\frac{1}{2}\beta^2\right)\norm{\what_{j}} \mathlinebreak\quad+ 2(\delta^2+\beta^2K)\norm{\centerr_{c,j}}^2+\frac{3}{2}\beta^2\zeta^2+\frac{1}{2}\sigma^2 \Bigg) \nonumber\\
    &\stackrel{(a)}{=} \frac{\epstau}{2}(1+\epstau) \norm{\xhat_{m\tau}}^2 \mathlinebreak\quad+2\tau\mu^2\sum_{\mathclap{j=m\tau}}^{i-1}\Big(9\beta^2\zeta^2+9\beta^2K\norm{\centerr_{c,j}}^2 +9\beta^2\norm{\xhat_{j}}^2+3\sigma^2 \Big)\mathlinebreak\quad+\frac{2\tau(1+\epstau)\mu^2}{1-\epstau}\sum_{\mathclap{j=m\tau}}^{i-1}\Bigg(3\left(2\delta^2+\frac{1}{2}\beta^2\right)\norm{\xhat_{j}} + 2(\delta^2\mathlinebreak\quad+\beta^2K)\norm{\centerr_{c,j}}^2+\frac{3}{2}\beta^2\zeta^2+\frac{1}{2}\sigma^2 \Bigg) \nonumber\\
    &= \frac{\epstau}{2}(1+\epstau)\norm{\xhat_{m\tau}}^2\mathlinebreak\quad +\left(\frac{3\tau(4\delta^2+\beta^2)(1+\epstau)\mu^2}{1-\epstau}+18\tau\beta^2\mu^2 \right)\sum_{\mathclap{j=m\tau}}^{i-1}\norm{\xhat_j}^2 \mathlinebreak\quad
    +\left(\frac{4\tau(\delta^2+\beta^2K)(1+\epstau)\mu^2}{1-\epstau} +18\tau\beta^2K\mu^2 \right)\sum_{\mathclap{j=m\tau}}^{i-1}\norm{\centerr_{c,j}}^2\mathlinebreak\quad+\left(18\tau\beta^2\mu^2+\frac{3\tau\beta^2(1+\epstau)\mu^2}{1-\epstau}\right)\sum_{\mathclap{j=m\tau}}^{i-1}\zeta^2\mathlinebreak\quad + \left(6\tau\mu^2+\frac{(1+\epstau)\mu^2}{1-\epstau}\right)\sum_{\mathclap{j=m\tau}}^{i-1}\sigma^2 \nonumber\\
    &\stackrel{(b)}{=} \frac{\epstau}{2}(1+\epstau)\norm{\xhat_{m\tau}}^2 \mathlinebreak\quad+\left(\frac{3\tau(4\delta^2+\beta^2)(1+\epstau)\mu^2}{1-\epstau}+18\tau\beta^2\mu^2 \right)\sum_{\mathclap{j=m\tau}}^{i-1}\norm{\xhat_j}^2 \mathlinebreak\quad
    +\left(\frac{4\tau(\delta^2+\beta^2K)(1+\epstau)\mu^2}{1-\epstau} +18\tau\beta^2K\mu^2 \right)\sum_{\mathclap{j=m\tau}}^{i-1}\norm{\centerr_{c,j}}^2\mathlinebreak\quad+2\tau\left(18\tau\beta^2\mu^2+\frac{3\tau\beta^2(1+\epstau)\mu^2}{1-\epstau}\right)\zeta^2 \mathlinebreak\quad+ 2\tau\left(6\tau\mu^2+\frac{(1+\epstau)\mu^2}{1-\epstau}\right)\sigma^2  
\end{align}
where in $(a)$ the summation indices were changed and we used the fact that $\norm{\what_i}\leq\norm{\xhat_i}^2$ based on the definition of the vector norm $\norm{\xhat_i}^2=\norm{\what_i}^2+\norm{\zhat_i}^2$. In $(b)$ we used the fact that ${\sum_{j=m\tau}^{i-1}1<2\tau}$. Taking the total expectation yields the statement of the lemma.

\vspace{-10pt}\section{Centroid Error Relation}\label{app:cent-err}
 \textcolor{black}{We briefly prove Lemma~\eqref{lemma:centroid_bound}, noting that the proof follows arguments in ~\cite{chen2015, vlaski21,sayed2014}. We begin by showing that $\onekronim \y_i=0$ from the definition of $y_i$ in~\eqref{eq:aug-dgm-y:y}:
 \begin{align}
     &\onekronim \y_i \nonumber\\&=\onekronim (\A_{i}\y_{i-1}-\mu\A_i(I-\A_{i-1})\gradappcal(\w_{i-1}) ) \nonumber\\
     &\stackrel{(a)}{=}\onekronim\y_{i-1}
     \nonumber \\
     &= \onekronim \y_0 \nonumber\\
     &\stackrel{(b)}{=} \onekronim(\g_0-\mu\A_0\gradappcal(\w_0)) \nonumber \\
     &\stackrel{(c)}{=} \onekronim \mu\gradappcal(\w_0)- \onekronim \mu\gradappcal(\w_0) \nonumber\\
     &= 0
 \end{align}
 where $(a)$ follows from ${\onekronim \Ahat_i=\onekronim}$ and ${\onekronim(I-\Ahat_i)=0}$, $(b)$ from the relation between $\g_i$ and $\y_i$ and $(c)$ from the initialization of $\g_i$.\\
We now consider the behavior of the centroid. Subtracting $w^o $ from both sides of~\eqref{eqn:centroid-evolution}:
\begin{align}
    &\centerr_{c,i} =\centerr_{c,i-1}-\mu\onekronim(\gradcal(\w_{i-1})+\s_{i}(\w_{i-1})) \nonumber\\
    &\expectnormfi{\centerr_{c,i}} \nonumber\\&=\mathbb{E}\Big[\Big\lVert\centerr_{c,i-1}-\mu\onekronim(\gradcal(\w_{i-1})\mathlinebreak\quad+\s_{i}(\w_{i-1}))\Big\rVert^2\Big\vert \w_{i-1}\Big] \nonumber\\
    &\stackrel{(a)}{=}\mathbb{E}\Big[\Big\lVert\centerr_{c,i-1}-\mu\onekronim\big(\gradcal(\w_{c,i-1})+\gradcal(\w_{i-1})\mathlinebreak~-\gradcal(\w_{c,i-1})+\s_{i}(\w_{i-1})\big)\Big\rVert^2\Big\vert \w_{i-1}\Big] \nonumber\\
    &\stackrel{~\eqref{eq:assump-gradnoise-unbias}}{=} \mathbb{E}\Big[\Big\lVert\centerr_{c,i-1}-\mu\onekronim\big(\gradcal(\w_{c,i-1})+\gradcal(\w_{i-1})\mathlinebreak~-\gradcal(\w_{c,i-1})\big)\Big\rVert^2\Big\vert \w_{i-1}\Big]\mathlinebreak~+\mu^2\mathbb{E}\Big[\Big\lVert\onekronim\s_{i}(\w_{i-1})\Big\rVert^2\Big\vert \w_{i-1}\Big] \nonumber\\
    &\stackrel{(b)}{=}  \mathbb{E}\Big[\Big\lVert\centerr_{c,i-1}+\mu\frac{1}{K}\sum_{k=1}^K\big(\nabla J_k(\wb_{c,i-1})-\nabla J_k(\boldsymbol{w}_{k,i-1})\mathlinebreak~-\nabla J_k(\wb_{c,i-1})\big)\Big\rVert^2\Big\vert \w_{i-1}\Big]\mathlinebreak~+  \mu^2\left( \frac{3\beta^2\zeta^2}{K}+3\beta^2\norm{\centerr_{c,i-1}}^2+\frac{3\beta^2}{K}\norm{\what_{i-1}}^2+\frac{\sigma^2}{K}\right) \nonumber\\
    &\stackrel{(c)}{=}  \frac{1}{\alpha}\norm{\centerr_{c,i-1}+\mu\nabla J(\wb_{c,i-1})}^2\mathlinebreak+\frac{\mu^2}{1-\alpha}\norm{\frac{1}{K}\sum_{k=1}^K\left(\nabla J_k(\boldsymbol{w}_{k,i-1})-\nabla J_k(\wb_{c,i-1})\right)}^2\mathlinebreak+   \mu^2\left(\frac{3\beta^2\zeta^2}{K}+3\beta^2\norm{\centerr_{c,i-1}}^2+\frac{3\beta^2}{K}\norm{\what_{i-1}}^2+\frac{\sigma^2}{K}\right)\nonumber\\
    &\leq \frac{1}{\alpha}\norm{\centerr_{c,i-1}+\mu\nabla J(\wb_{c,i-1})}^2+\frac{\mu^2\delta^2}{K(1-\alpha)}\norm{\what_{i-1}}^2\mathlinebreak+   \mu^2\left(\frac{3\beta^2\zeta^2}{K}+3\beta^2\norm{\centerr_{c,i-1}}^2+\frac{3\beta^2}{K}\norm{\what_{i-1}}^2+\frac{\sigma^2}{K}\right)\label{eqn:cent-err-proof-part}
\end{align}
where in $(a)$ we added and subtracted ${-\mu\onekronim\gradcal(\w_{c,i-1})}$, in $(b)$ we used~\eqref{eq:avg-stacked-gradnoise-var} and in $(c)$ we applied Jensen's inequality. The first term in~\eqref{eqn:cent-err-proof-part} is equal to $(1-2\nu\mu+\delta^2\mu^2)\norm{\centerr_{c,i-1}}^2$ (see ~\cite{chen2015, vlaski21,sayed2014}). Setting $\alpha=\sqrt{1-2\nu\mu+\delta^2\mu^2}$: 
\begin{align}
    &\expectnormfi{\centerr_{c,i}} \nonumber \\&\leq \sqrt{1-2\nu\mu+\delta^2\mu^2}\norm{\centerr_{c,i-1}}^2+\frac{\delta^2\mu^2}{K(\nu\mu-\frac{1}{2}\delta^2\mu^2)}\norm{\what_{i-1}}^2 \mathlinebreak+ \mu^2\left(\frac{3\beta^2\zeta^2}{K}+3\beta^2\norm{\centerr_{c,i-1}}^2+\frac{3\beta^2}{K}\norm{\what_{i-1}}^2+\frac{\sigma^2}{K}\right) \nonumber\\
    &\leq \left(\sqrt{1-2\nu\mu+\delta^2\mu^2}+\mathcal{O}(\mu^2)\right)\norm{\centerr_{c,i-1}}^2 \mathlinebreak+ \left(\frac{2\delta^2\mu}{K(2\nu-\delta^2\mu)}+\mathcal{O}(\mu^2) \right)\norm{\xhat_{i-1}}^2 + \frac{3\beta^2\mu^2}{K}\zeta^2+\frac{\mu^2}{K}\sigma^2 \label{eqn:cent-err-proof-part2}
\end{align}
where in the last line we used the fact that ${\norm{\what_{i-1}}\leq\norm{\xhat_{i-1}}}$ based on the definition of the vector norm.
The first term in the expression is bound by:
\begin{align}
    \sqrt{1-2\nu\mu+\delta^2\mu^2}\leq \sqrt{1-\nu\mu}
    &\leq 1-\frac{\nu\mu}{2}
\end{align}
where the first inequality holds under the condition that ${\mu\leq\frac{\nu}{\delta^2}}$.
Taking the total expectation of~\eqref{eqn:cent-err-proof-part2} returns the bound in the lemma}.

\section{Main Result}\label{app:main-res}

\subsection{Derivation of Maximum Centroid Error in~\eqref{eqn:max-cent-err}}
Consider now the centroid error from Lemma~\eqref{lemma:centroid_bound}. We will iterate over the centroid error for $\tau$ iterations and then consider the maximum error:

\begin{align}
    &\expectnorm{\centerr_{c,i}}\leq\alpha_1\expectnorm{\centerr_{c,i-1}} + \alpha_2\expectnorm{\xhat_{i-1}}+\alpha_3 \nonumber\\
    ~&\leq \alpha_1^2\expectnorm{\centerr_{c,i-2}} + \alpha_2\alpha_1\expectnorm{\xhat_{i-2}}+\alpha_2\expectnorm{\xhat_{i-1}}\mathlinebreak~~+\alpha_3\left(1+\alpha_1\right) \nonumber\\
    % &= \alpha_1^2\expectnorm{\centerr_{c,i-2}} + \alpha_2\sum_{j=i-2}^{i-1}\left(\alpha_1^{i-j-1}\expectnorm{\xhat_j}\right)+\alpha_3\sum_{j=0}^1\alpha_1^j \\
    ~&\leq \alpha_1^{\tau}\expectnorm{\centerr_{c,i-\tau}}+\alpha_2\sum_{j=i-\tau}^{i-1}\left(\alpha_1^{i-j-1}\expectnorm{\xhat_j}\right)+\alpha_3\sum_{j=0}^{\tau-1}\alpha_1^j \nonumber\\
    ~&\stackrel{\left(a\right)}{\leq} \alpha_1^{\tau}\expectnorm{\centerr_{c,i-\tau}}+\alpha_2\sum_{j=i-\tau}^{i-1}\expectnorm{\xhat_j}+\alpha_3\sum_{j=0}^{\tau-1}\alpha_1^j \nonumber\\
    ~&\stackrel{\left(b\right)}{\leq} \alpha_1^{\tau}\expectnorm{\centerr_{c,i-\tau}}+\alpha_3\sum_{j=0}^{\tau-1}\alpha_1^j\mathlinebreak~~+\alpha_2\tau\Big(\max_{j\in\mathcal{S}_{m+1}}\expectnorm{\xhat_j} + ~\max_{j\in\mathcal{S}_{m+2}}\expectnorm{\xhat_j}\Big) \nonumber\\
    ~&\stackrel{\left(c\right)}{\leq} \alpha_1^{\tau}  \max_{i\in\mathcal{S}_{m+1}}\expectnorm{\centerr_{c,i}}+\alpha_2\tau \max_{j\in\mathcal{S}_{m+1}}\expectnorm{\xhat_j}\mathlinebreak~~ + \alpha_2\tau\max_{j\in\mathcal{S}_{m+2}}\expectnorm{\xhat_j} +\alpha_3\sum_{j=0}^{\tau-1}\alpha_1^j \nonumber\\
    ~&\stackrel{\left(d\right)}{\leq} \alpha_1^\tau\omega_{m+1}+\alpha_2\tau\chi_{m+1}+\alpha_2\tau\chi_{m+2} + \tau\alpha_3  \label{eq:main-res-part-proof}
\end{align}
where in $\left(a\right)$ we used that $\alpha_1<1$, in $\left(b\right)$ we replaced the consensus error term $\expectnorm{\xhat_j}$ in the summation with the maximum consensus error terms for the two FTC periods over which the summation runs, in $\left(c\right)$ we replaced the centroid error term with the maximum centroid error in the previous period, and in $\left(d\right)$ we updated the notation to use $\omega_m$ and $\chi_m$ and used the fact that $\alpha_1\leq1$. In the last line we took the maximum over both sides of the inequality for $\left(m+1\right)\tau\leq i\leq \left(m+2\right)\tau-1$. Taking the maximum of~\eqref{eq:main-res-part-proof} over $i\in\mathcal{S}_{m+2}$:
\begin{align}
          \omega_{m+2}  &\leq \alpha_1^\tau\omega_{m+1}+\alpha_2\tau\chi_{m+1}+\alpha_2\tau\chi_{m+2} + \tau\alpha_3
\end{align}

\subsection{Derivation of Coupled Recursion in~\eqref{eqn:coupled-recursion}}

We begin by substituting in~\eqref{eqn:max-cent-err} for $\omega_{m+2}$ in~\eqref{eqn:max-cons-err}:
\begin{align}
    \chi_{m+2} &\leq \left(\theta_1 + \tau\theta_2+\tau\left(\tau-1\right)\alpha_2\theta_3\right)\chi_{m+1}\mathlinebreak\quad+\left(\left(\tau-1\right)\theta_2+\tau\left(\tau-1\right)\alpha_2\theta_3\right)\chi_{m+2}\mathlinebreak\quad+\left(\tau+\alpha_1^\tau\left(\tau-1\right)\right)\theta_3\omega_{m+1}+\theta_4 + \tau\left(\tau-1\right)\alpha_3\theta_3
    \end{align}
    Rearranging the $\chi_{m+2}$ terms to the left:
    \begin{align}
        &(1-\left(\tau-1\right)\theta_2-\tau\left(\tau-1\right)\alpha_2\theta_3)\chi_{m+2} \nonumber\\ &\quad\leq \left(\theta_1 + \tau\theta_2+\tau\left(\tau-1\right)\alpha_2\theta_3\right)\chi_{m+1}\mathlinebreak\qquad+\left(\tau+\alpha_1^\tau\left(\tau-1\right)\right)\theta_3\omega_{m+1}+\theta_4 + \tau\left(\tau-1\right)\alpha_3\theta_3\nonumber\\
    &\chi_{m+2} \leq \frac{3}{2}\left(\theta_1 + \tau\theta_2+\tau\left(\tau-1\right)\alpha_2\theta_3\right)\chi_{m+1}\mathlinebreak\qquad+\frac{3}{2}\left(\tau+\alpha_1^\tau\left(\tau-1\right)\right)\theta_3\omega_{m+1}+\frac{3}{2}\theta_4 + \frac{3}{2}\tau\left(\tau-1\right)\alpha_3\theta_3 \label{eqn:chi_m2-decoupled}
\end{align}
where the last line follows by setting the step-size such that ${1-\left(\tau-1\right)\theta_2-\tau\left(\tau-1\right)\alpha_2\theta_3>\frac{2}{3}}$. Substituting~\eqref{eqn:chi_m2-decoupled} into~\eqref{eqn:max-cent-err} yields:
\begin{align}
    \omega_{m+2} &\leq \left(\alpha_1^\tau+\frac{3}{2}\alpha_2\theta_3\tau\left(\tau+\alpha_1^{\tau}\left(\tau-1\right)\right)\right)\omega_{m+2}  \mathlinebreak\quad+\alpha_2\tau\left(1+\frac{3}{2}\theta_1+\frac{3}{2}\tau\theta_2+\frac{3}{2}\tau\left(\tau-1\right)\alpha_2\theta_3\right)\chi_{m+1}\mathlinebreak\quad+\tau\alpha_3+\frac{3}{2}\alpha_2\theta_4+\frac{3}{2}\alpha_2\alpha_3\theta_3\tau^2\left(\tau-1\right)
\end{align}
Noting that ${\alpha_2=\mathcal{O}\left(\mu\right)}$, ${\theta_2=\mathcal{O}\left(\mu^2\right)}$, ${\theta_3=\mathcal{O}\left(\mu^2\right)}$, ${\theta_4=\mathcal{O}\left(\mu^2\right)}$ we are left with:
\begin{align}
    \omega_{m+2} &\leq \alpha_1^\tau\omega_{m+2} + \alpha_2\tau\left(1+\frac{3}{2}\theta_1\right)\chi_{m+1}+\tau\alpha_3 +\mathcal{O}\left(\mu^3\right) \nonumber\\
    \chi_{m+2} &\leq \frac{3}{2}\left(\theta_1 + \tau\theta_2\right)\chi_{m+1}+\frac{3}{2}\tau\theta_3\omega_{m+1}+\frac{3}{2}\theta_4 +\mathcal{O}\left(\mu^3\right)
\end{align}

\subsection{Proof of Theorem~\ref{thm:main-res}}
Beginning with the coupled recursion in~\eqref{eqn:coupled-recursion}:
\begin{align}\label{eqn:coupled-recursion_repeat}
    \begin{bmatrix}
        \omega_m \\
        \chi_m
    \end{bmatrix} &\leq \begin{bmatrix}
        \alpha_1^\tau & \alpha_2\tau\left(1+\frac{3}{2}\theta_1\right) \\
        \frac{3}{2}\tau\theta_3 & \frac{3}{2}\left(\theta_1+\tau\theta_2\right)
    \end{bmatrix} \begin{bmatrix}
        \omega_{m-1} \\
        \chi_{m-1}
    \end{bmatrix}+\begin{bmatrix}
        \tau\alpha_3 \\
        \frac{3}{2}\theta_4 
    \end{bmatrix}+\mathcal{O}\left(\mu^3\right) \nonumber\\
    &\leq \underbrace{\begin{bmatrix}
        1-\frac{\tau\nu\mu}{4} & \alpha_2\tau\left(1+\frac{3}{2}\theta_1\right) \\
        \frac{3}{2}\tau\theta_3 & \frac{3}{2}\left(\theta_1+\tau\theta_2\right)
    \end{bmatrix}}_{\triangleq H} \begin{bmatrix}
        \omega_{m-1} \\
        \chi_{m-1}
    \end{bmatrix}+\underbrace{\begin{bmatrix}
        \tau\alpha_3 \\
        \frac{3}{2}\theta_4 
    \end{bmatrix}}_{\triangleq p}+\mathcal{O}\left(\mu^3\right)
\end{align} 
where the last line follows from ${\left(1-x\right)^\tau\leq 1-\frac{\tau x}{2}}$ for ${x\leq\frac{1}{\tau}}$ which holds here for $\mu\leq\frac{2}{\nu\tau}$.

We will bound the spectral radius of $H$ in the subsequent lemma to establish the convergence of~\eqref{eqn:coupled-recursion_repeat}. First we note that the eigenvalues of $H$ are given by $ {\lambda_{1,2} = \frac{ H_{11} + H_{22} \pm\sqrt{\left( H_{11} - H_{22} \right)^2+4 H_{12} H_{21} }}{2}}$. Thus:
\begin{align}
    \rho\left(H\right) &= \frac{ H_{11} + H_{22} +\sqrt{\left( H_{11} - H_{22} \right)^2+4 H_{12} H_{21} }}{2} \label{eq:rho_H_exact}
    \end{align}

    \textcolor{black}{\begin{lemma} The spectral radius of $H$ in~\eqref{eqn:coupled-recursion_repeat} is bounded by:
    \begin{align}
    \rho(H)\leq \frac{1+H_{11}+H_{22}(1-H_{11})}{2} \label{eq:rho_H_approx} 
    \end{align}
    for $\mu\leq\frac{\nu\sqrt{K}(1-\epstau)}{72\tau\delta\sqrt{\delta^2+\beta^2}}$.
    \end{lemma}
    \begin{proof}
    Rearranging the terms in~\eqref{eq:rho_H_exact} and~\eqref{eq:rho_H_approx}:
    \begin{align}
        &\sqrt{\left( H_{11} - H_{22} \right)^2+4 H_{12} H_{21}  }\leq 1-H_{11}H_{22} \\
        &4H_{12}H_{21}\stackrel{(a)}{=} (1-H_{11}^2)(1-H_{22}^2)~\label{eq:rho_H_deriv1}:
    \end{align}
    where $(a)$ follows from squaring both sides and rearranging the terms. Since $0<H_{11},H_{22}<1$, a sufficient condition for~\eqref{eq:rho_H_deriv1} is:
        \begin{align}
        H_{12}H_{21}&\leq \frac{1}{4}(1-H_{11})(1-H_{22})
    \end{align}
    Substituting back in the values for $H_{k\ell}$ from~\eqref{eqn:coupled-recursion_repeat}:
    \begin{align}
        \frac{3}{2}\tau^2\alpha_2\theta_3\left(1+\frac{3}{2}\theta_1\right)&\leq\frac{\tau\nu\mu}{16}\left( 1-\frac{3}{2}\theta_1-\frac{3}{2}\tau\theta_2\right)
    \end{align}
    Noting that $1+\frac{3}{2}\theta_1\leq \frac{5}{2}$ and $1-\frac{3}{2}\tau\theta_2\geq \frac{1}{16}$ for ${\mu\leq\frac{\sqrt{1-\epstau}}{8\tau\sqrt{4\delta^2+\beta^2}}}$. Substituting $\theta_3$ under the condition that $\epstau\leq\frac{3}{4}$:
    % \begin{align}
        % \tau^2\alpha_2\theta_3\leq\frac{\tau\nu\mu}{60}\left(\frac{15}{16}-\frac{3}{4}\epstau(1+\epstau)\right) \nonumber\\
        % \theta_3\leq\frac{\nu^2K}{120\tau\delta^2}\left(\frac{15}{16}-\frac{3}{4}\epstau(1+\epstau)\right)
    % \end{align}
    % And substituting in $\theta_3$ with $\epstau\leq\frac{3}{4}$:
    \begin{align}
    \mu\leq\frac{\nu\sqrt{K}(1-\epstau)}{72\tau\delta\sqrt{\delta^2+\beta^2}}
    \end{align}
    \end{proof}}
    Thus, 
    \begin{align}
        \rho(H)&\leq \frac{1+H_{11}+H_{22}(1-H_{11})}{2} \nonumber \\ 
        &= 1-\frac{\tau\nu\mu}{8}+\frac{3\tau\nu\mu}{32}\cdot\epstau(1+\epstau)+\mathcal{O}(\mu^3)
    \end{align}

Given that $\rho\left(H\right)<1$ and reconsidering the system in~\eqref{eqn:coupled-recursion_repeat}:
\begin{align}
    \begin{bmatrix}
        \omega_m \\
        \chi_m
    \end{bmatrix} &\leq H^{m-1} \begin{bmatrix}
        \omega_1 \\
        \chi_1
    \end{bmatrix} + \sum_{j=0}^{m-1}H^j p \nonumber\\
    &\leq H^{m-1} \begin{bmatrix}
        \omega_1 \\
        \chi_1
    \end{bmatrix} + \left(I-H\right)^{-1}p \nonumber\\
    \left \lVert \begin{bmatrix}
        \omega_m \\
        \chi_m
    \end{bmatrix} \right\rVert_1  &\leq \left\lVert H^{m-1} \begin{bmatrix}
        \omega_1 \\
        \chi_1
    \end{bmatrix} + \left(I-H\right)^{-1}p \right\rVert_1 \nonumber\\
    \omega_m+\chi_m &\leq \left\lVert H^{m-1}\right\rVert_1 \left\lVert\begin{bmatrix}
        \omega_1 \\
        \chi_1
    \end{bmatrix} \right\rVert_1 + \left\lVert\left(I-H\right)^{-1}p \right\rVert_1 \label{eqn:msd-deriv}
\end{align}
where:
\begin{align}
    &\left(I-H\right)^{-1} \nonumber \\ \quad &= \frac{1}{\det\left(I-H\right)}\begin{bmatrix}
        1-\frac{3}{2}\left(\theta_1+\tau\theta_2\right) & \alpha_2\tau\left(1+\frac{3}{2}\theta_1\right) \\
        \frac{3}{2}\tau\theta_3 & \frac{\tau\nu\mu}{4} 
    \end{bmatrix} \label{eq:main-res-I-H}
\end{align}
For the determinant of $I-H$:
\begin{align}
    % \det\left(I-H\right) &=\left(1-\alpha_1^\tau\right)\left(1-\frac{3}{2}\left(\theta_1+\tau\theta_2\right)\right)-\mathcal{O}\left(\mu^3\right) \nonumber\\
    % &\geq\left(1-\alpha_1\right)\left(1-\frac{3}{2}\left(\theta_1+\tau\theta_2\right)\right)-\mathcal{O}\left(\mu^3\right) \nonumber\\
    \det\left(I-H\right) &= \frac{\tau\nu\mu}{4}\left(1-\frac{3}{2}\left(\theta_1+\tau\theta_2\right)\right)-\mathcal{O}\left(\mu^3\right) 
\end{align}
 Thus, returning to~\eqref{eq:main-res-I-H} and considering the entry-wise inequality:
\begin{align}
    \left(I-H\right)^{-1} &\leq \begin{bmatrix}
        \frac{1}{\mu\tau\nu} & \frac{4\alpha_2(1+\frac{3}{2}\theta_1)}{\mu\nu\left(1-\frac{3}{2}\left(\theta_1+\tau\theta_2\right)\right)} \nonumber \\
        \frac{12\theta_3}{\mu\nu\left(1-\frac{3}{2}\left(\theta_1+\tau\theta_2\right)\right)} & \frac{1}{1-\frac{3}{2}\left(\theta_1+\tau\theta_2\right)} 
    \end{bmatrix}  \\
    &\leq \begin{bmatrix} \frac{4}{\mu\tau\nu} & \frac{240\alpha_2}{\mu\nu\left(1-\epstau\right)} \\
        \frac{360\theta_3}{\mu\nu\left(1-\epstau\right)} & \frac{30}{1-\epstau } \end{bmatrix} %condition for last line is $\mu^2\leq\frac{2-\gamma\epstau\left(1+\epstau\right)}{24\gama\tau^2\left(1+\epstau\right)\left(4\delta^2+\beta^2\right)}$ but this is already covered by some of the other conditions
    \end{align}
    % where the last line holds for ${1-\frac{3}{2}\theta_1-\frac{3\tau}{2}\theta_2\geq \frac{1}{2}-\frac{3}{4}\theta_1}$ under the conditions that $\mu\leq \frac{\sqrt{0.758-\epstau}}{6\sqrt{3}\tau\beta}$ and ${\mu\leq\frac{0.758-\epstau}{3\sqrt{2}\tau\sqrt{4\delta^2+\beta^2}}}$.
    % Finally, we get:
    % \begin{align}
    %     &\left(I-H\right)^{-1}p \nonumber \\&\leq \begin{bmatrix*}[l]
    %         \frac{\mu\tau}{\nu K}\sigma^2 + \frac{24\tau^2\mu}{\nu\left(1-\frac{3}{2}\theta_1\right)}\sigma^2 +\frac{4\tau\left(1+\epstau\right)\mu}{\nu\left(1-\epstau\right)\left(1-\frac{3}{2}\theta_1\right)}\sigma^2 \\ \qquad\qquad+\frac{3\beta^2\tau\mu}{K\nu}\zeta^2 +\frac{12\tau^2\beta^2\left(1+\epstau\right)\mu}{\nu\left(1-\epstau\right)\left(1-\frac{3}{2}\theta_1\right)}\zeta^2+\frac{72\tau^2\beta^2\mu}{\nu\left(1-\frac{3}{2}\theta_1\right)}\zeta^2 \\
    %         \frac{36\tau^2\mu^2}{1-\frac{3}{2}\theta_1}\sigma^2+\frac{6\tau^2\left(1+\epstau\right)\mu^2}{\left(1-\epstau\right)\left(1-\frac{3}{2}\theta_1\right)}\sigma^2 \\ \qquad\qquad +\frac{18\tau^3\beta^2\left(1+\epstau\right)\mu^2}{\left(1-\epstau\right)\left(1-\frac{3}{2}\theta_1\right)}\zeta^2+\frac{108\tau^3\beta^2\mu^2}{1-\frac{3}{2}\theta_1}\zeta^2
    %     \end{bmatrix*}
    % \end{align}
    where the last line holds since ${1+\frac{3}{2}\theta_1\leq 2}$ and ${1-\frac{3}{2}\theta_1-\frac{3\tau}{2}\theta_2\geq \frac{1}{30}-\frac{\epstau}{30}}$ under the conditions that $\mu\leq \frac{\sqrt{(1-0.75\epstau)(1-\epstau)}}{5\tau\sqrt{4\delta^2+\beta^2}}$.
    Finally, we get:
    \begin{align}
        % &\left(I-H\right)^{-1}p \nonumber \\&\leq \begin{bmatrix*}[l]
        %     \frac{\mu\tau}{\nu K}\sigma^2 + \frac{24\tau^2\mu}{\nu\left(1-\frac{3}{2}\theta_1\right)}\sigma^2 +\frac{4\tau\left(1+\epstau\right)\mu}{\nu\left(1-\epstau\right)\left(1-\frac{3}{2}\theta_1\right)}\sigma^2 \\ \qquad\qquad+\frac{3\beta^2\tau\mu}{K\nu}\zeta^2 +\frac{12\tau^2\beta^2\left(1+\epstau\right)\mu}{\nu\left(1-\epstau\right)\left(1-\frac{3}{2}\theta_1\right)}\zeta^2+\frac{72\tau^2\beta^2\mu}{\nu\left(1-\frac{3}{2}\theta_1\right)}\zeta^2 \\
        %     \frac{36\tau^2\mu^2}{1-\frac{3}{2}\theta_1}\sigma^2+\frac{6\tau^2\left(1+\epstau\right)\mu^2}{\left(1-\epstau\right)\left(1-\frac{3}{2}\theta_1\right)}\sigma^2 \\ \qquad\qquad +\frac{18\tau^3\beta^2\left(1+\epstau\right)\mu^2}{\left(1-\epstau\right)\left(1-\frac{3}{2}\theta_1\right)}\zeta^2+\frac{108\tau^3\beta^2\mu^2}{1-\frac{3}{2}\theta_1}\zeta^2
        % \end{bmatrix*}
        &\left(I-H\right)^{-1}p \nonumber \\&\leq \begin{bmatrix*}[l]
            \frac{4\mu}{\nu K}\sigma^2 + \frac{8640\mu^2\tau^2\delta^2}{\nu^2 K\left(1-\epstau\right)}\sigma^2 +\frac{1440\mu^2\tau\left(1+\epstau\right)\delta^2}{\nu^2 K\left(1-\epstau\right)^2}\sigma^2 +\frac{12\mu\beta^2}{\nu K}\zeta^2\\ \qquad +\frac{4320\mu^2\tau^2\left(1+\epstau\right)\delta^2\beta^2}{\nu^2 K\left(1-\epstau\right)^2}\zeta^2+\frac{25920\mu^2\tau^2\delta^2\beta^2}{\nu^2 K\left(1-\epstau\right)}\zeta^2 + \mathcal{O}\left(\mu^3\right) \\
            \frac{540\mu^2\tau^2}{1-\epstau}\sigma^2+\frac{90\mu^2\tau\left(1+\epstau\right)}{\left(1-\epstau\right)^2}\sigma^2 \\ \qquad\qquad +\frac{1620\mu^2\tau^2\beta^2}{1-\epstau}\zeta^2+\frac{270\mu^2\tau^2(1+\epstau)\beta^2}{(1-\epstau)^2}\zeta^2+ \mathcal{O}\left(\mu^3\right)
        \end{bmatrix*}
    \end{align}

The MSD is then:
\begin{align}
    &\frac{1}{K}\norm{ \w_i - \one\otimes w^o}^2 \nonumber\\
    &\quad=  \frac{1}{K}\norm{ \w_i -\w_{c,i}+\w_{c,i}-\one\otimes w^o }^2 
    \end{align}
    The maximum MSD over the $m$-th FTC sequence is thus:
    \begin{align}
    &\quad\leq \frac{2}{K}\norm{ \w_i -\w_{c,i}}^2 + \frac{2}{K}\norm{\w_{c,i}-\one\otimes w^o}^2 \nonumber\\
    &\quad\leq 2 \norm{\xhat_i}^2+2\norm{\centerr_{c,i}}^2 \nonumber\\
    &\max_{i\in\mathcal{S}_m} \expectnorm{ \w_i - \one\otimes w^o} \leq 2\omega_m + 2\chi_m \nonumber\\
    % &= 2\left \lVert \begin{bmatrix}
    %     \omega_m \\
    %     \chi_m
    % \end{bmatrix} \right\rVert_1 \nonumber\\
    &\quad\stackrel{\eqref{eqn:msd-deriv}}{\leq} \lVert H^{m-1} \rVert_1 \left(2\omega_1+2\chi_1\right)+2\left\lVert\left(I-H\right)^{-1}p \right\rVert_1 \nonumber\\
    &\quad\leq \lVert H^{m-1} \rVert_2 \left(2\sqrt{2}\omega_1+2\sqrt{2}\chi_1\right)+2\left\lVert\left(I-H\right)^{-1}p \right\rVert_1 \nonumber\\
    &\quad\leq \rho(H)^{m-1} \kappa\left(2\sqrt{2}\omega_1+2\sqrt{2}\chi_1\right)+2\left\lVert\left(I-H\right)^{-1}p \right\rVert_1 \nonumber\\
    % &\quad\leq \gamma^{m-1}c_1+\mathcal{O}\Bigg( \frac{\mu\tau\sigma^2}{\nu K} +\frac{\tau^2\left(1+\epstau\right)\mu\sigma^2}{\nu\left(1-\epstau\right)\left(4-3\epstau\left(1+\epstau\right)\right)}\mathlinebreak\qquad +\frac{\beta^2\tau\mu\zeta^2}{\nu K}+ \frac{\beta^2\tau^2\left(1+\epstau\right)\mu\zeta^2}{\nu\left(1-\epstau\right)\left(4-3\epstau\left(1+\epstau\right)\right)}\Bigg)+\mathcal{O}\left(\mu^2\right)
    &\quad\leq \gamma^{m-1}c_1+\mathcal{O}\Bigg( \frac{\mu\sigma^2}{\nu K}+\frac{\mu\beta^2\zeta^2}{\nu K}  \mathlinebreak\qquad+\frac{\mu^2\tau^2\left(1+\epstau\right)(\sigma^2+\beta^2\zeta^2)}{\nu^2 K\left(1-\epstau\right)^2}\mathlinebreak\qquad + \frac{\mu^2\tau^2(1+\epstau)(\sigma^2+\beta^2\zeta^2}{(1-\epstau)^2}\Bigg)+\mathcal{O}\left(\mu^2\right)
\end{align}
where $c_1=2\kappa\sqrt{2}(\omega_1+\chi_1)$, ${\kappa\triangleq \norm{H}\norm{H^{-1}}}$ is the spectral condition number of H, and ${\gamma=1-\frac{\tau\nu\mu}{8}+\frac{3\tau\nu\mu\epstau(1+\epstau)}{32}+\mathcal{O}(\mu^3)}$.
% $$\gamma= \begin{cases} 
%       1-\frac{\tau\nu\mu}{8} & \epstau=0 \\
%       1-\frac{\tau\nu\mu}{4}+\frac{1+\epstau}{2} & 0\leq \epstau\leq \frac{3}{4} 
%    \end{cases}$$

\clearpage
\section*{Supplementary Material: FTC Rules}\label{app:ftc-rules}

\subsection*{Path Graph}
We provide an FTC sequence for a path graph where the number of nodes is a power of 2. It is based on a recursive formula. Note that $A^{(K)}_i$ denotes the $i^{\textrm{th}}$ matrix in the FTC sequence with $K$ nodes. 
\begin{align}\label{mat:4path}
    &A^{(4)}_{1}=\begin{bmatrix}[1.35]
\frac{1}{2} & \frac{1}{2} & 0 & 0 \\
\frac{1}{2} & \frac{1}{2} & 0 & 0 \\
0 & 0 & \frac{1}{2} & \frac{1}{2} \\
0 & 0 & \frac{1}{2} & \frac{1}{2} 
\end{bmatrix}, \quad A^{(4)}_{2}=\begin{bmatrix}[1.35]
1 & 0 & 0 & 0 \\
0 & 0 & 1 & 0 \\
0 & 1 & 0 & 0 \\
0 & 0 & 0 & 1 
\end{bmatrix}, \nonumber\\
&\qquad\qquad\qquad A^{(4)}_{3}=\begin{bmatrix}[1.35]
\frac{1}{2} & \frac{1}{2} & 0 & 0 \\
\frac{1}{2} & \frac{1}{2} & 0 & 0 \\
0 & 0 & \frac{1}{2} & \frac{1}{2} \\
0 & 0 & \frac{1}{2} & \frac{1}{2} 
\end{bmatrix}
\end{align}

Additionally we define the $2\times 2$ exchange matrix as \footnotesize{${F\triangleq \begin{bmatrix}
0 & 1 \\
1 & 0 
\end{bmatrix}  }$}\normalsize.\\

The FTC sequence for the path graph with $K$ nodes is then a sequence of $\tau={K-1}$ matrices given by:
    
\begin{align}
&\text{for} \qquad 1\leq i\leq \frac{K}{2}-1: \nonumber \\
    &\qquad A_{i}^{(K)} = \begin{bmatrix}
A_{i}^{\left(\frac{K}{2}\right)} & 0 \\
0 & A_i^{\left(\frac{K}{2}\right)}
\end{bmatrix} \nonumber \\
&\text{for} \qquad \frac{K}{2}\leq i \leq K-2:\nonumber\\
&\qquad A_{i}^{(K)} = \begin{bmatrix}
I_{K-i-1} &  &  &  &  \\
 & F &  &  &  \\
0 &  & \ddots &  & 0 \\
 &  &  & F &  \\
 &  &  &  & I_{K-i-1} 
\end{bmatrix}  \nonumber\\
&\text{finally:} \qquad \nonumber \\
&\qquad A_{K-1}^{(K)} = \begin{bmatrix}
\frac{1}{2}\one_2\one_2^\mathsf{T}  &  & 0 \\
 & \ddots &  \\
0 &  & \frac{1}{2}\one_2\one_2^\mathsf{T} 
\end{bmatrix}  
\end{align}

\subsection*{Ring Graphs}
We provide an FTC sequence for a ring graph where the number of nodes $K$ is a power of 2. First we define the $2\times2$ exchange matrix: \footnotesize{${F\triangleq \begin{bmatrix}
0 & 1 \\
1 & 0 
\end{bmatrix}  }$}\normalsize and the  $K\times K$ matrices:
\begin{align}
    &B_1 \triangleq \begin{bmatrix}
    \frac{1}{2}\one_2\one_2^\mathsf{T}  & 0 & \dots  & 0 & 0\\
    0 & \frac{1}{2}\one_2\one_2^\mathsf{T} &\dots & 0 & 0 \\
     & & \ddots &  \\
    0 & 0 & \dots &\frac{1}{2}\one_2\one_2^\mathsf{T} & 0 \\
    0 & 0 & \dots  & 0 & \frac{1}{2}\one_2\one_2^\mathsf{T} \\
\end{bmatrix}  \nonumber\\
    &B_2 \triangleq \begin{bmatrix}
    \frac{1}{2}  & 0 & \dots  & 0 & \frac{1}{2} \\
    0 & \frac{1}{2}\one_2\one_2^\mathsf{T} &\dots & 0 & 0 \\
     & & \ddots &  \\
    0 & 0 & \dots &\frac{1}{2}\one_2\one_2^\mathsf{T} & 0 \\
    \frac{1}{2}  & 0 & \dots  & 0 & \frac{1}{2} \\
\end{bmatrix}\nonumber\\
    & D_1 \triangleq \begin{bmatrix}
    F  & 0 & \dots & 0 & 0 \\
    0 & F  & \dots & 0  & 0 \\
     & & \ddots & &  \\
    0 & 0 & \dots &  F & 0 \\ 
    0 &  0 & \dots & 0 &F 
\end{bmatrix} \nonumber\\
    & D_2 \triangleq \begin{bmatrix}
    0  & 0 & \dots & 0 & 1 \\
    0 & F  & \dots & 0  & 0 \\
     & & \ddots & &  \\
    0 & 0 & \dots &  F & 0 \\ 
    1 &  0 & \dots & 0 & 0
\end{bmatrix} 
\end{align}

The FTC sequence is then defined by:
\begin{align}
    A_i = \begin{cases}
        B_1 & \textrm{if} \quad i=1 \\
        B_2 & \textrm{if} \quad i\&(i-1)=0~ \textrm{and}~ i\neq1 \\
        D_1 & \textrm{if} \quad i\&(i-1)\neq0~ \textrm{and}~ i\%2=1 \\
        D_2 & \textrm{if} \quad i\&(i-1)\neq0~ \textrm{and}~ i\%2=0 \\
    \end{cases}
\end{align}
for $1\leq i\leq \frac{K}{2}$ and where $i\&j$ is the bitwise logical and between numbers $i$ and $j$.

\vfill
}

\vfill

\end{document}